\documentclass[reqno]{amsart}
\usepackage{graphicx} 
\usepackage[utf8]{inputenc}
\usepackage{amsmath}
\usepackage{fullpage}
\usepackage{amssymb, amsfonts}
\usepackage{amsthm,mathtools,marvosym}
\usepackage{amsmath,amsfonts,amssymb,mathrsfs,amsthm,bbm,comment}
\usepackage{tikz}
\usepackage{tikz-cd}
\usepackage{mathrsfs}
\usepackage{hyperref}
\usepackage{cleveref}
\usepackage{mathscinet}
\usepackage{placeins}
\usepackage{enumitem}

\theoremstyle{plain}
\newtheorem{theorem}{Theorem}[section]
\newtheorem{lemma}[theorem]{Lemma}
\newtheorem{proposition}[theorem]{Proposition}
\newtheorem{corollary}[theorem]{Corollary}

\newtheorem*{theorem*}{Theorem}
\newtheorem*{corollary*}{Corollary}

\theoremstyle{definition}
\newtheorem{definition}[theorem]{Definition}
\newtheorem{example}[theorem]{Example}
\newtheorem{remark}[theorem]{Remark}

\theoremstyle{remark}

\newcommand{\newword}[1]{\emph{\textbf{#1}}}

\newlength\cellsize 

\savebox6{%
\begin{picture}(13,13)
\put(0,0){\line(1,0){13}}
\put(0,0){\line(0,1){13}}
\put(13,0){\line(0,1){13}}
\put(0,13){\line(1,0){13}}
\end{picture}}

\newcommand\cellify[1]{\def\thearg{#1}\def\nothing{}%
\ifx\thearg\nothing\vrule width0pt height\cellsize depth0pt%
  \else\hbox to 0pt{\usebox6\hss}\fi%
  \vbox to 13\unitlength{\vss\hbox to 13\unitlength{\hss$#1$\hss}\vss}}

\newcommand\tableau[1]{\vtop{\let\\=\cr
\setlength\baselineskip{-10000pt}
\setlength\lineskiplimit{10000pt}
\setlength\lineskip{0pt}
\halign{&\cellify{##}\cr#1\crcr}}}

\savebox4{%
\begin{picture}(20,20)
\put(0,0){\line(2,0){20}}
\put(0,0){\line(0,2){20}}
\put(20,0){\line(0,2){20}}
\put(0,20){\line(2,0){20}}
\end{picture}}

\newcommand\Bcellify[1]{\def\thearg{#1}\def\nothing{}%
\ifx\thearg\nothing\vrule width0pt height\cellsize depth0pt%
  \else\hbox to 0pt{\usebox4\hss}\fi%
  \vbox to 20\unitlength{\vss\hbox to 20\unitlength{\hss$#1$\hss}\vss}}

\newcommand\Btableau[1]{\vtop{\let\\=\cr
\setlength\baselineskip{-10000pt}
\setlength\lineskiplimit{10000pt}
\setlength\lineskip{0pt}
\halign{&\Bcellify{##}\cr#1\crcr}}}

\savebox2{%
\begin{picture}(13,13)
\put(-.2,-.2){\tikz{\filldraw [fill=gray!40!white, draw=white] (0,0) rectangle (.45,.45);}}
\put(0,0){\line(1,0){13}}
\put(0,0){\line(0,1){13}}
\put(13,0){\line(0,1){13}}
\put(0,13){\line(1,0){13}}
\end{picture}}

\newcommand\scellify[1]{\def\thearg{#1}\def\nothing{}%
\ifx\thearg\nothing\vrule width0pt height\cellsize depth0pt%
  \else\hbox to 0pt{\usebox2\hss}\fi%
  \vbox to 13\unitlength{\vss\hbox to 13\unitlength{\hss$#1$\hss}\vss}}

\newcommand\spintab[1]{\vtop{\let\\=\cr
\setlength\baselineskip{-10000pt}
\setlength\lineskiplimit{10000pt}
\setlength\lineskip{0pt}
\halign{&\scellify{##}\cr#1\crcr}}}

\savebox9{%
\begin{picture}(20,20)
\put(-.2,-.2){\tikz{\filldraw [fill=gray!40!white, draw=white] (0,0) rectangle (.7,.7);}}
\put(0,0){\line(20,0){20}}
\put(0,0){\line(0,2){20}}
\put(20,0){\line(0,2){20}}
\put(0,20){\line(2,0){20}}
\end{picture}}

\newcommand\Bscellify[1]{\def\thearg{{#1}}\def\nothing{}%
\ifx\thearg\nothing\vrule width0pt height\cellsize depth0pt%
  \else\hbox to 0pt{\usebox9\hss}\fi%
  \vbox to 20\unitlength{\vss\hbox to 20\unitlength{\hss$#1$\hss}\vss}}

\newcommand\Bspintab[1]{\vtop{\let\\=\cr
\setlength\baselineskip{-10000pt}
\setlength\lineskiplimit{10000pt}
\setlength\lineskip{0pt}
\halign{&\Bscellify{##}\cr#1\crcr}}}

\input xy
\usepackage[all]{xy}
\xyoption{line}
\xyoption{arrow}
\xyoption{color}
\SelectTips{cm}{}

\usepackage{tikz}
\usetikzlibrary{decorations.markings}
\usetikzlibrary{decorations.pathreplacing}
\newcommand{\hackcenter}[1]{
 \xy (0,0)*{#1}; \endxy}

\tikzstyle directed=[postaction={decorate,decoration={markings,
    mark=at position #1 with {\arrow{>}}}}]
\tikzstyle rdirected=[postaction={decorate,decoration={markings,
    mark=at position #1 with {\arrow{<}}}}]

\usetikzlibrary{calc}
\usepackage{relsize}

\tikzset{fontscale/.style = {font=\relsize{#1}}
    }

\usepackage[normalem]{ulem}
\usepackage[colorinlistoftodos]{todonotes}
\newcommand{\g}{\mathfrak{g}}
\newcommand{\fb}{\mathfrak{b}}
\newcommand{\Z}{\mathbbm{Z}}
\newcommand{\B}{\mathcal{B}}
\newcommand{\C}{\mathcal{C}}
\newcommand{\cO}{\mathcal{O}}
\newcommand{\E}{\mathcal{E}}
\newcommand{\F}{\mathcal{F}}
\newcommand{\V}{\mathcal{V}}
\newcommand{\D}{\mathbbm{D}}
\newcommand\Hom{\mathrm{Hom}}
\newcommand\wt{\mathrm{wt}}

\newcommand{\vt}{\text{\Aries}}

\newcommand{\aB}{\mathsf{B}}
\newcommand{\aC}{\mathsf{C}}

\newcommand{\KNb}[2]{\mathrm{KN}^B_{#1}{#2}}
\newcommand{\KNc}[2]{\mathrm{KN}^C_{#1}{#2}}
\newcommand{\Split}{\mathrm{split}}
\newcommand{\evac}{\operatorname{evac}}

\newcommand{\rect}{\operatorname{rect}}
\newcount\cols
{\catcode`,=\active\catcode`|=\active
 \gdef\Young#1{\hbox{$\vcenter
 {\mathcode`,="8000\mathcode`|="8000
  \def,{\global\advance\cols by 1 &}%
  \def|{\cr
        \multispan{\the\cols}\hrulefill\cr
        &\global\cols=2 }%
  \offinterlineskip\everycr{}\tabskip=0pt
  \dimen0=\ht\strutbox \advance\dimen0 by \dp\strutbox
  \halign
   {\vrule height \ht\strutbox depth \dp\strutbox##
    &&\hbox to \dimen0{\hss$##$\hss}\vrule\cr
    \noalign{\hrule}&\global\cols=2 #1\crcr
    \multispan{\the\cols}\hrulefill\cr%
   }
 }$}}
\gdef\Skew(#1:#2){\hbox{$\vcenter
{\mathcode`,="8000\mathcode`|="8000
  \dimen0=\ht\strutbox \advance\dimen0 by \dp\strutbox
  \def\boxbeg{\vbox
    \bgroup\hrule\kern-0.4pt\hbox to\dimen0\bgroup\strut\vrule\hss$}%
  \def\boxend{$\hss\egroup\hrule\egroup}%
  \def,{\boxend\boxbeg}%
  \def|##1:{\boxend\vrule\egroup\nointerlineskip\kern-0.4pt
    \moveright##1\dimen0\hbox\bgroup\boxbeg}%
  \def\\##1\\##2:{\boxend\vrule\egroup\nointerlineskip\kern-0.4pt
    \kern ##1\dimen0\moveright##2\dimen0\hbox\bgroup\boxbeg}%
  \moveright#1\dimen0\hbox\bgroup\boxbeg#2\boxend\vrule\egroup
 }$}}
}

    \newcommand{\oneonetwo}{
        \spintab{1\\2}\tableau{1}
    }

    \newcommand{\twoonetwo}{
        \spintab{1\\2}\tableau{2}
    }
    \newcommand{\vtwoonetwo}{
        \tableau{1&2&2\\2}
    }

    \newcommand{\oneonebartwo}{
        \spintab{1\\\bar{2}}\tableau{1}
    }
    \newcommand{\voneonebartwo}{
        \tableau{1&1&1\\ \bar{2}}
    }

    \newcommand{\zeroonetwo}{
        \spintab{1\\2}\tableau{0}
    }
    \newcommand{\vzeroonetwo}{
        \tableau{1& 2& \bar{2}\\2}
    }

    \newcommand{\twoonebartwo}{
        \spintab{1\\\bar{2}}\tableau{2}
    }
    \newcommand{\vtwoonebartwo}{
        \tableau{1&2&2\\ \bar{2}}
    }

    \newcommand{\bartwoonetwo}{
        \spintab{1\\2}\tableau{\bar{2}}
    }
    \newcommand{\vbartwoonetwo}{
        \tableau{1&\bar{2}&\bar{2}\\2}
    }

    \newcommand{\twotwobarone}{
        \spintab{2\\\bar{1}}\tableau{2}
    }
    \newcommand{\vtwotwobarone}{
        \tableau{2&2&2\\\bar 1}
    }

    \newcommand{\zeroonebartwo}{
        \spintab{1\\\bar{2}}\tableau{0}
    }
    \newcommand{\vzeroonebartwo}{
        \tableau{1&2&\bar{2}\\\bar{2}}
    }

    \newcommand{\baroneonetwo}{
        \spintab{1\\2}\tableau{\bar{1}}
    }
    \newcommand{\vbaroneonetwo}{
        \tableau{1&\bar{1}&\bar{1}\\2}
    }

    \newcommand{\bartwoonebartwo}{
        \spintab{1\\\bar{2}}\tableau{\bar{2}}
    }
    \newcommand{\vbartwoonebartwo}{
        \tableau{1&\bar{2}&\bar{2}\\\bar{2}}
    }

    \newcommand{\zerotwobarone}{
        \spintab{2\\\bar{1}}\tableau{0}
    }
    \newcommand{\vzerotwobarone}{
        \tableau{2&2&\bar{2}\\\bar{1}}
    }

    \newcommand{\baroneonebartwo}{
        \spintab{1\\\bar{2}}\tableau{\bar{1}}
    }
    \newcommand{\vbaroneonebartwo}{
        \tableau{1&\bar{1}&\bar{1}\\\bar{2}}
    }

    \newcommand{\bartwotwobarone}{
        \spintab{2\\\bar{1}}\tableau{\bar{2}}
    }
    \newcommand{\vbartwotwobarone}{
        \tableau{2&\bar{2}&\bar{2}\\\bar{1}}
    }

    \newcommand{\baronetwobarone}{
        \spintab{2\\\bar{1}}\tableau{\bar{1}}
    }
    \newcommand{\vbaronetwobarone}{
        \tableau{2&\bar{1}&\bar{1}\\\bar{1}}
    }

    \newcommand{\bartwobartwobarone}{
        \spintab{\bar{2}\\\bar{1}}\tableau{\bar{2}}
    }
    \newcommand{\vbartwobartwobarone}{
        \tableau{\bar{2}&\bar{2}&\bar{2}\\\bar{1}}
    }

    \newcommand{\baronebartwobarone}{
        \spintab{\bar{2}\\\bar{1}}\tableau{\bar{1}}
    }
    \newcommand{\vbaronebartwobarone}{
        \tableau{\bar2&\bar{1}&\bar{1}\\\bar{1}}
    }
    
\title{Keys and evacuation via virtualization 
}

\author[O. Azenhas]{Olga Azenhas}
\address{ University of Coimbra, CMUC, Department of Mathematics, Portugal}
\email{oazenhas@mat.uc.pt}

\author[N. Gonz\'alez]{Nicolle Gonz\'alez}
\address{Department of Mathematics,  University of California Berkeley, CA 94720-3840, USA}
\email{nicolle@math.berkeley.edu}

\author[D. Huang]{Daoji Huang}
\address{Institute for Advanced Study, Princeton, NJ 08540, USA}
\email{dhuang@ias.edu}

\author[J. Torres]{Jacinta Torres}
\address{Institute of Mathematics,  Jagiellonian University in Kraków, Poland}
\email{jacinta.torres@uj.edu.pl}

\begin{document}

\begin{abstract}
In this paper, we study the relation between the key map and virtualization of crystals. Namely, we prove that virtualization between crystals in any two finite Cartan types commutes with the left and right key maps, thus embedding Demazure crystals and atoms correspondingly. In particular, this implies that the key map in any finite Cartan type can be reduced to the key map in a simply-laced type, provided an appropriate virtualization exists, generalizing the work of Azenhas--Santos. As an application, we study these maps in the context of orthogonal Kashiwara--Nakashima tableaux and show that the virtualizations from type B into C considered independently by Fujita and Pappe--Pfannerer--Schilling--Simone coincide with the splitting map of De Concini and Lecouvey. As a consequence, this enables us to give a new and purely combinatorial definition of orthogonal evacuation.
\end{abstract}

\maketitle
\setcounter{secnumdepth}{4}
\setcounter{tocdepth}{1}
\section*{Introduction}
Let \(\mathfrak{g}\) be a finite-dimensional Lie algebra and \(\lambda\) a dominant integral weight. Kashiwara \cite{Kas90} (and independently Lusztig \cite{lusztig94}) showed that, to every irreducible finite-dimensional complex \(\mathfrak{g}\)-representation \(V(\lambda)\) of highest weight \(\lambda\), one can associate the so-called \newword{crystal graph} \(\mathcal{B}(\lambda)\). This graph consists of vertices given by the crystal (or canonical) basis of \(V(\lambda)\) and edges corresponding to certain deformations of the Chevalley generators. These weighted, directed, colored graphs serve as combinatorial skeletons of the representations \(V(\lambda)\), encoding many of their fundamental properties while being computationally easier to work with.

\newword{Demazure modules} are certain Borel modules that generalize the highest weight modules of \(\mathfrak{g}\). Geometrically, Demazure modules correspond to certain spaces of global sections of a suitable line bundle on a Schubert variety embedded in a flag variety \cite{demazure74}. Indeed, it is well-known that the theory of crystals is closely related to standard monomial theory on flag varieties, with a natural correspondence between Schubert varieties and Demazure modules \cite{kumar, littelmann1995paths, litlakmag98, Lakshlittelmann2002LS}. In particular, it was shown by Littelmann and Kashiwara \cite{kashdemazure, littelmann95conj} that, for any \(w\) in the Weyl group \(W\) of \(\mathfrak{g}\) and dominant integral weight \(\lambda\), the Demazure module \(V_{w}(\lambda)\) has an associated \newword{Demazure crystal} \(\B_{w}(\lambda)\), which arises as a particular induced subgraph of \(\B(\lambda)\). Thus, for fixed \(w \in W\), it is natural to inquire whether a given vertex \(b \in \B(\lambda)\) belongs to the Demazure crystal \(\B_{w}(\lambda)\).

The answer to this question is provided by the \newword{key map}. This map associates to each vertex \(b \in \B(\lambda)\) a pair of elements in the \(W\)-orbit of the highest weight element in \(\B(\lambda)\), called the \newword{right} and \newword{left key} of \(b\). The right key indicates the smallest possible Demazure crystal containing the vertex \(b\). Highly related are \newword{Demazure atoms}, which are subsets of \(\B(\lambda)\) that consist of all the elements in \(\B(\lambda)\) whose right key is the same\footnote{There is an analogous dual relation between the left key of \(b\) and the so-called \emph{opposite Demazure crystal} and \emph{opposite Demazure atoms}.}. The study of Demazure crystals and atoms is an very active topic of research; see, for instance, \cite{samarmon,assdrgonz23,agh24, AGL, azenhas2024virtualkeys, aval, brubaker,buciumascrimshaw22, baker00bn,fourier2013demazure, fulas, fujita18,fujita22,lenarthershposet2017,jacon2020keys,kouno20,kundu2024saturation, kus2016twisted, Lakshlittelmann2002LS, lam19, SarahMasonAtoms, pr99, sa21a, sa21b}. Historically, much of the relevance of the key map stems from standard monomial theory, with appearances in essentially every crystal model. Specifically, the right and left keys were designed to encapsulate the Lakshmibai--Seshadri minimal and maximal defining chains in standard monomial theory \cite{LaSchu90keys, reinershimozono97, seshadri2016introduction, seshadri2016, fatemeh21}. In the Littelmann crystal model of Lakshmibai--Seshadri paths for \(\B(\lambda)\) \cite{littelmann1994LS}, the initial and final directions of a path exhibit the right key and left key, respectively. Similarly, within the Lenart--Postnikov alcove path model \cite{lenart2008combinatorial, lenart2007combinatorics}, the right and left key also play a role.

Due to its prevalence in the study of crystals, the effective computation of the key map has captured the interest of many mathematicians over the last few decades. For \(\mathfrak{g} = \mathfrak{sl}(n+1, \mathbb{C})\), there is a classical algorithm based on jeu de taquin for semi-standard Young tableaux due to Lascoux--Sch\"utzenberger \cite[Proposition 5.2]{LaSchu90keys}, with many other algorithms for type \(A_n\) developed in \cite{aval,brubaker,kushwaha,SarahMasonAtoms,willisdirect}. More generally, Jacon--Lecouvey also gave a recursive procedure for any Kac--Moody Lie algebras using combinatorial R-matrices, which generalizes the original procedure of Lascoux--Sch\"utzenberger \cite{jacon2020keys}. A different recursive algorithm was also provided by Hersh--Lenart \cite{lenarthershposet2017} using the edge-colored poset-theoretic structure of the crystal. Despite these advances, direct combinatorial methods for computing right and left keys in general types are not very well studied. In the case of reverse King and Sundaram tableaux of types \(B_{n}\) and \(C_{n}\), there is a procedure due to Buciumas--Scrimshaw \cite{buciumascrimshaw22} which computes the right keys and atoms in terms of certain quasi-solvable lattice models. Additionally, not long ago, two combinatorial methods on Kashiwara--Nakashima tableaux of type \(C_{n}\) were constructed by Santos \cite{sa21a,sa21b}.

Of interest to us, a more recent third method of computing right and left keys in type \(C_n\) was developed by Azenhas--Santos \cite{azenhas2024virtualkeys} using a process called \newword{virtualization}. This approach allowed them to translate the computations into classical and well-understood type \(A_n\) combinatorics. Abstractly, virtualization is a method introduced by Kashiwara in \cite{Kashsimilar96}, whereby a highest weight \(\mathfrak{g}\)-crystal is embedded inside some highest weight \(\tilde{\mathfrak{g}}\)-crystal, provided the Dynkin diagram of \(\mathfrak{g}\) can be obtained from the Dynkin diagram of \(\tilde{\mathfrak{g}}\) via a diagram folding. The image of such an embedding, equipped with an induced crystal structure, is termed a \newword{virtual crystal}. Virtualization maps have been constructed in various models for most finite Dynkin types corresponding to various diagram foldings \cite{baker2000zero, schiscrivirtual15, pappe2023promotion, fujita18, pan2018virtualization, nakajimamonomial}. The main goal of this paper is a generalization of the methods developed by Azenhas--Santos \cite{azenhas2024virtualkeys} to any finite Cartan type. Since virtualization enables any finite type to be transformed into some simply-laced type (see Figure \ref{fig:gamma}), our results imply that calculating the key map for any finite type can always be simplified to a calculation within a simply-laced type where the combinatorial methods are better developed.

\subsection*{Dilation and Virtualization} A distinguished family of virtualizations from any Cartan type to itself are Kashiwara's \newword{crystal dilations} \cite{Kashsimilar96}. These maps are given by the natural embeddings \(\Theta_m:\B(\lambda) \hookrightarrow \B(m\lambda) \subset \B(\lambda)^{\otimes m}\) for any positive integer \(m\). Crystal dilations are central to the computation of the right and left keys since, for \(m\) large enough, iterated applications of \(\Theta_m\) to any given \(b \in \B(\lambda)\) decompose the image of \(b\) as a tensor product of extremal weight vectors. Although Kashiwara proved \cite[Prop. 8.3.2]{kashiwara002cours} that, for any \(b \in \B(\lambda)\), there exists an \(m\) such that \(\Theta_m(b) = b_{w\lambda} \otimes b' \otimes b_{w'\lambda}\) with \(w \geq w'\) in strong Bruhat order and \(b_{w\lambda}\) and \(b_{w'\lambda}\) extremal vectors, his proof does not offer tight bounds on how large a fixed \(m\) must be in order for the result to hold for every \(b \in \B(\lambda)\) simultaneously. In Theorem \ref{thm:extremal endpoints}, we enhance Kashiwara's findings and prove the following:

\begin{theorem*}[\ref{thm:extremal endpoints}]
Let $m \in \mathbb{N}$. 
For all $b \in \B(\lambda)$, there exists $b' \in \B(\lambda)^{\otimes (m-2)}$ and fixed $w\geq w' \in W$ such that 
\[\Theta_m(b) = b_{w\lambda}\otimes b' \otimes b_{w'\lambda} \quad \text{ if and only if } \quad m \geq \ell\]
where $\ell=\max \{ \text{length}(\rho) \; | \; \rho \text{ is an i-string in } \B(\lambda)\text{ for some } i \in I\}.$
\end{theorem*}

By definition, the right and left key are precisely the extremal weight vectors \(b_{w\lambda}\) and \(b_{w'\lambda}\) in the decomposition of \(\Theta_m(b)\). Hence, Theorem \ref{thm:extremal endpoints} provides a tight bound for choosing \(m\), which uniformly works in computing the right and left key of every vertex in \(\B(\lambda)\).

This method of calculating the keys of a given element, however, is not very efficient. Thus, we turn to virtualization as a method of translating the computation of keys in a given type onto another where direct computation may be more effective. Our main result, Theorem \ref{thm:keys-virt-commute}, provides a complete solution to this problem in every finite Dynkin type by fully generalizing the method studied by Azenhas--Santos \cite{azenhas2024virtualkeys}.

\begin{theorem*}[\ref{thm:keys-virt-commute}] 
Let $\psi:X \to Y$ be a Dynkin diagram embedding with virtualization map $\vt:\B(\lambda) \to \tilde{\B}(\psi(\lambda))$. Then, for any $b \in \B(\lambda)$: 
\[\vt(K^+(b))=K^+(\vt(b)) \text{ and } \vt(K^-(b))=K^-(\vt(b)).\] 
Thus, virtualization embeds Demazure crystals and atoms correspondingly, so that for any $w \in W^X$ we have
\[ \B_w(\lambda) \overset{\vt}{\hookrightarrow} \tilde{\B}_{\psi(w)}(\psi(\lambda)) \qquad \text{ and } \qquad \mathring {\B}_{w}(\lambda)\overset{\vt}{\hookrightarrow} \mathring{\tilde{\B}}_{\psi(w)}(\psi(\lambda)).\]
\end{theorem*}

\subsection*{Evacuation and Applications to Type B} The Lusztig--Sch\"utzenberger involution is defined as the unique set involution \(\xi_{\B(\lambda)}: \B(\lambda) \rightarrow \B(\lambda)\), induced by the automorphism of the Dynkin diagram of \(\mathfrak{g}\) given by left multiplication by the longest element in \(W\). The process of combinatorially producing the image \(\xi(b)\) for a given \(b \in \B(\lambda)\) directly, without needing to compute the entire crystal, is known as \newword{evacuation}.  
Historically, evacuation in type \(A_n\) was defined by Sch\"utzenberger for semistandard Young tableaux \cite{Schutzenberger1963} and later proven by Berenstein--Zelevinsky \cite{berenstein1996canonical} to coincide with the Lusztig involution on canonical bases. A similar algorithm using symplectic jeu de taquin was recently developed by Santos for type \(C_n\) \cite{sa21a}. Evacuation has also been considered in non-tableaux models, such as the alcove path model \cite{lenart2007combinatorics} or the Lakshmibai--Seshadri path model \cite{littelmann1994LS}. In particular, in each of these models, it was shown that evacuation exchanges right and left keys.

In Theorem \ref{thm:Keys-Schutz-commute}, we provide a new elementary and model-independent proof that the Lusztig--Sch\"utzenberger involution exchanges the right and left keys. Namely, we first prove that the Lusztig--Sch\"utzenberger involution splits across tensor factors in multiplicity-one isotypic components (see Proposition \ref{prop:commutor-trivial}). From here, using the fact that dilation maps \(\B(\lambda)\) onto a component of \(\B(\lambda)^{\otimes m}\) of multiplicity one, the result follows. Since the Lusztig--Sch\"utzenberger involution also commutes with any virtualization \cite[Thm. 4]{torres2024virtual}, our results enable us to give a combinatorial formulation for orthogonal evacuation in the context of Kashiwara--Nakashima tableaux, as we now explain.

In \cite{fujita18}, Fujita defined a virtualization map \(\vt_{BC}\) from type \(B_n\) into \(C_n\) in their study of Newton--Okounkov polytopes. In a different context and independently, in \cite{pappe2023promotion}, Pappe--Pfannerer--Schilling--Simone defined what turned out to be the same virtualization from \(B_n\) into \(C_n\). In Theorem \ref{thm:virt-column} and Corollary \ref{cor:virt-nospin}, we show that the virtualization considered by Fujita and Pappe et al. on orthogonal Kashiwara--Nakashima tableaux coincides with a combinatorial operation studied by Lecouvey in \cite{lecouvey2002schensted, lecouvey2003schensted, lecouvey2007combinatorics}, known as \newword{splitting}. That is, we prove:
\begin{theorem*}[\ref{cor:virt-nospin}]
    For any orthogonal Kashiwara--Nakashima tableau $T$, we have $\vt_{BC}(T)= \Split(T)$.
    \end{theorem*}
Consequently, we obtain that the virtualization $\vt_{BC}$ intertwines with rectification in types $B$ and $C$.
\begin{corollary*}[\ref{cor:virt-rect-commute}] For any orthogonal skew Kashiwara--Nakashima tableau $T$, we have
\[
\vt_{BC} \circ \rect_B(T) = \rect_C \circ \vt_{BC}(T).
    \]
\end{corollary*}
Lastly, combining Theorems \ref{thm:Keys-Schutz-commute} and Corollary \ref{cor:virt-nospin}, we obtain that orthogonal evacuation $\evac_B$ is given by the composition
\begin{equation}\label{eq:evacBintro}
\evac_B:= (\Split)^{-1} \evac_C \Split,
\end{equation}
where $\evac_C$ is the symplectic evacuation procedure on type $C_n$ Kashiwara--Nakashima tableaux defined by Santos in \cite{sa21b}. In particular, \eqref{eq:evacBintro} provides the first combinatorial algorithm for directly computing the Lusztig--Sch\"utzenberger involution in type $B_n$ on any given vertex. 

Finally, in Proposition \ref{prop:key} we give a complete combinatorial characterization of when an orthogonal Kashiwara--Nakashima tableau is an extremal weight vector, and thus a key.

\subsection*{Structure of the Paper:}
In Section \ref{sec:crystals} we recall the necessary background on normal crystals, their tensor products, the Lusztig--Sch\"utzenberger involution, and virtualization. Section \ref{sec:demazure} contains the majority of the theoretical results. In it, we discuss Demazure crystal and atoms, begin the study of dilation and its properties, and prove Theorem \ref{thm:extremal endpoints}. After introducing the right and left key maps, we prove Theorem  \ref{thm:Keys-Schutz-commute}. We then explore how virtualization interacts with dilation and extremal vectors, from which we deduce the main result, Theorem \ref{thm:keys-virt-commute}. We then move towards computational applications in Section \ref{applicationsb}. Here we recall the crystal structure on orthogonal and symplectic Kashiwara--Nakashima tableaux, symplectic jeu de taquin, and then use this to prove Theorem \ref{thm:virt-column} and Corollaries \ref{cor:virt-nospin} and \ref{cor:virt-rect-commute}. We culminate in Section \ref{sec:examples} with various detailed examples displaying how the primary theorems in the paper can be used to compute right and left keys for orthogonal Kashiwara--Nakashima tableaux.

\section*{Acknowledgements}
The authors thank the Banff International Research Station for their hospitality, as well as the organizers of the workshop ``Community in Algebraic Combinatorics", where this project was born. We also thank Cédric Lecouvey and Anne Schilling for discussions, and  Travis Scrimshaw for helpful comments.
O. A. was partially supported by the Centre for Mathematics of the University of Coimbra (funded by the Portuguese Government through FCT/MCTES, DOI 10.54499/UIDB/00324/2020).
J.T. was supported by the grant SONATA NCN UMO-2021/43/D/ST1/02290 and partially supported by the grant MAESTRO NCN UMO-2019/34/A/ST1/00263.

\section{Crystal Graphs and the Lusztig--Sch\"utzenberger Involution}
\label{sec:crystals}

\subsection{Normal Crystals}\label{cartandata}
We review Kashiwara's theory of $\g$-crystals but refer the reader to \cite{Kas90,Kas91,banffcrystal, kashiwara002cours,humphreys97,bjorner2005combinatorics} for details. 
\\

Let $\g$ be a finite complex semisimple Lie algebra with integral weight lattice $P$,
simple roots $\alpha_i \in P$ with $i \in I$, simple coroots $\alpha_i^\vee \in P^\vee=\Hom_{\Z}(P, \Z)$, and canonical pairing $\langle \cdot, \cdot \rangle: P^\vee \times P \to \Z$. The fundamental weights $\omega_{i}$ and coweights $\omega^{\vee}_{i}$ are defined via $\langle \alpha^{\vee}_{i}, \omega_{j}\rangle =\langle \omega^{\vee}_{j}, \alpha_{i}\rangle = \delta_{ij}$. We denote the corresponding Weyl group as $W$; it is generated as a Coxeter group by  simple reflections  $s_{i}, i \in I$ and relations defined by the associated Dynkin diagram. For $u,v \in W$, we say that $u \leq v$ in the \textit{strong Bruhat order} if, for every reduced expression for $v$, there exists a subexpression that is a reduced expression for $u$. 

\begin{definition}A (normal) \newword{$\g$-crystal}  is a nonempty finite set $\B$  with maps:
\[
\wt: \B \to P, \qquad \varepsilon_i, \varphi_i: \B \to \Z, \qquad e_i, f_i: \B \to \B \sqcup \{ 0\},
\]
where $0\notin \B$ is an auxiliary symbol, subject to the following conditions for all $i \in I$ and $b, b' \in \B$ :
\begin{itemize}
    \item[(C1)] $\varphi_i(b)- \varepsilon_i(b) = \langle \alpha_i^\vee, \wt(b) \rangle$;
    \item[(C2)]  $ \wt(e_i(b)) = \wt(b) + \alpha_i$ if $e_i(b) \in \B$ and  $ \wt(f_i(b)) = \wt(b) - \alpha_i$ if $f_i(b) \in \B$;
    \item[(C3)] $b' = e_i(b)$ if and only if $b = f_i(b')$;
    \item[(C4)]  $\varepsilon_i(b) =$ max$\{ k \geq 0 \vert e_i^k(b) \in \B\}$ and $\varphi_i(b) =$ max$\{ k \geq 0 \vert f_i^k(b) \in \B\}$.
    \end{itemize}
We call the maps $\varepsilon_i$ and $\varphi_i$ the \emph{string operators}, $\wt$ the \emph{weight map}, and $e_i$ and $f_i$ the \emph{crystal operators}.
\end{definition}

For any $i \in I$, an \emph{$i$-string} of length $k$ is any subset of the form $\{f_i^n(b)\neq 0\;|\; n\geq 0\} \subset \B$ for some $b\in\B$ satisfying $e_i(b)=0$ and $\varphi_i(b)=k$.

\begin{definition}
Given any $\B, \C \in \g$-crystals, a map $\phi: \B \to \C \cup \{ 0\}$ is a \newword{crystal morphism} if for all $b \in \B$ and $\phi(b) \in \C$ and any $i \in I$ the following conditions hold:
\begin{itemize}
\item[$(a)$] $\wt(b) = \wt(\phi(b))$,
\item[$(b)$] $\varepsilon_i(b) =\varepsilon_i(\phi(b))$
\item[$(c)$] $\varphi_i(b) = \varphi_i(\phi(b)),$ 
\item[$(d)$] if $e_i(b) \in \B$ and $\phi(e_i(b)) \in \C$ then $e_i(\phi(b))= \phi(e_i(b))$, and
\item[$(e)$] if $f_i(b) \in \B$ and $\phi(f_i(b)) \in \C$ then $f_i(\phi(b))= \phi(f_i(b))$.
\end{itemize}
Moreover, we say $\phi$ is an \emph{isomomorphism} (resp. epimorphism, monomorphism) if the underlying set map $\phi:\B \to \C$ is a bijection (resp. injective, surjective). 
\end{definition}

The category of $\g$-crystals is endowed with a monoidal structure that is compatible with the tensor product structure of $U_q(\g)$. Thus, given $\B_1, \B_2 \in \g$-crystals the \emph{tensor product} $\B_1 \otimes \B_2$ is defined as the set
 \[\B_1 \otimes \B_2 = \{ b_1 \otimes b_2 \mid b_1\in \B_1 \mbox{ and } b_2\in \B_2 \}\]
with $\wt$ and string operators given by
\begin{align}\label{tensorproduct1}
  \wt(b_1 \otimes b_2) & = \wt(b_1) + \wt(b_2), \\
  \varepsilon_i(b_1 \otimes b_2) & = \max( \varepsilon_i(b_1), \varepsilon_i(b_2) - \wt_i(b_1) ),\nonumber\\
  \varphi_i(b_1 \otimes b_2) & = \max( \varphi_i(b_2), \varphi_i(b_1) + \wt_i(b_2) ),\nonumber
\end{align}
where $\wt_i(b) = \langle \alpha^{\vee}_i , \wt(b) \rangle$, and with crystal operators defined by \footnote{  This article follows the Kashiwara convention for crystal tensor products which differs from the Bump-Schilling convention \cite{bump2017crystal} by exchanging the order of the factors.} 
\begin{align}\label{tensorproduct2}
  e_i(b_1 \otimes b_2) &=
  \begin{cases}
    e_i(b_1) \otimes b_2 & \mbox{if } \varphi_i(b_1) \geq \varepsilon_i(b_2) , \\
    b_1 \otimes e_i(b_2) & \mbox{if } \varphi_i(b_1)< \varepsilon_i(b_2)  ;
  \end{cases}\\\label{tensorproduct2a}
  f_i(b_1 \otimes b_2) &= 
  \begin{cases}
    f_i(b_1) \otimes b_2 & \mbox{if } \varphi_i(b_1)>\varepsilon_i(b_2), \\
    b_1 \otimes f_i(b_2) & \mbox{if } \varphi_i(b_1) \le \varepsilon_i(b_2)  .
  \end{cases}
  \end{align}

\begin{definition}\label{def:monoid}
Let $\E$ and $\F$ be the monoids generated by $\{e_i\}_{i\in I}$ and $\{f_i\}_{i \in I}$, respectively. We say that an element $b \in \B$ is a \newword{highest weight vector} (resp. \newword{lowest weight vector}) if $\E \{b\} = \{b\}$ (resp. $\F\{b\} = \{b\}$). 
\end{definition}

It is known that the irreducible finite-dimensional integrable highest weight modules $V(\lambda)$ of $U_q(\g)$ are indexed by the set of dominant weights $P^+$. Thus, given $\lambda \in P^+$ we denote by $\B(\lambda)$ the normal crystal associated to $V(\lambda)$, whose highest weight vector $b_\lambda$ is the unique element in $\B(\lambda)$ satisfying the property that $\wt(b_\lambda)=\lambda$, 
\[ \E\{b_\lambda\} = \{b_\lambda\} \qquad \text{and} \qquad \F\{b_\lambda\} = \B(\lambda).\]

\begin{theorem}[\cite{banffcrystal, kashiwara002cours}]
  For $\lambda,\mu\in P^+$, $\B(\lambda) \otimes \B(\mu)$ is the crystal for $V(\lambda)\otimes V(\mu)$.
  \label{thm:func}
\end{theorem}
 The following decomposition formula follows from the definition of the tensor product. 

 \begin{proposition}
 \label{decomposition}
For  $\lambda, \mu \in P^+$,
\[\displaystyle \B (\lambda) \otimes \B (\mu)\simeq
\bigoplus_{\begin{smallmatrix}b\in\B(\mu)\\\varepsilon_i(b)\le \varphi_i(b_\lambda)\, \;i\in I \end{smallmatrix}}
\B(\lambda+\wt(b)).\]
\end{proposition}

The action of crystal operators on $\B_1\otimes\cdots\otimes \B_r$ can be computed using \emph{signature rules}. As noted in \cite[p.23]{bump2017crystal}, these rules are obtained by assigning to each factor $b_j$ in $b_1 \otimes \dots \otimes b_r$ the sign $(-)^{\varepsilon_i(b_j)} (+)^{\varphi_i(b_j)}$  and then successively bracketing any pair of the form $(+-)$ until all unbracketed symbols are of the form $(-)^{a} (+)^{b}$. It then follows that
\begin{equation}
    \varepsilon_i(b_1 \otimes \dots \otimes b_r) = a \hbox{ and }\varphi_i(b_1 \otimes \dots \otimes b_r) = b,\label{tensorproduct}
\end{equation}
so that $e_i$ will act on the factor associated to the rightmost unbracketed $(-)$ and $f_i$ will act on the factor associated to the leftmost unbracketed $(+)$ .

\subsection{Virtual Crystals}
\label{virtualdefs}
For any Dynkin diagram $D$, denote by  $P_{D}$ the corresponding integral weight lattice and by $\omega^D_i$ the corresponding fundamental weights.
Let $X$ and $Y$ be two Dynkin diagrams and let $aut$  be an automorphism of $Y$ such that distinct nodes of $Y$ in the same $aut$-orbit are not connected by an edge. We say there is an embedding $\psi: X \hookrightarrow Y$ if there exists a bijection $\Psi: X \rightarrow Y/aut$, which preserves the edges, inducing a map 
\[P_{X} \rightarrow P_{Y}\]
given by the assignment 
\[\omega^{X}_{i} \mapsto \sum_{j \in { \Psi}(i)} \gamma_{i} (\omega^{Y}_{j}), \]
with $\gamma_i$ given as in the  Table \ref{fig:gamma}.

Consequently, we have a natural embedding of the Weyl group $W^X$ into $W^Y$, identifying $W^X$ with the set of elements $\widetilde{W}^X$ in $W^Y$ that are fixed under the Dynkin symmetry:
\begin{align}\label{weylvirtual}
W^X \cong \widetilde{W}^X:= \langle
\Pi_{j\in \psi(i)} \tilde{s}_j \; |\; i \in I^X\rangle \subset W^Y= \langle \tilde{s}_j \; | \; j \in I^Y \rangle,
\end{align}

\noindent via the group isomorphism $s_i \mapsto \Pi_{j\in \psi(i)} \tilde{s}_j$. 
We abuse notation and use $\psi$ to also denote the induced maps on weight lattices, Weyl groups, and indices $\psi:I^X \to I^Y$. In particular, $\psi $ preserves the strong Bruhat order and reduced expressions for elements.

\begin{table}[ht]
\[
\begin{array}{ccccc} 
      \textbf{X} & \textbf{Y} & \ \gamma_{i}&~\Psi\\
\hline 
 C_n & A_{2n-1} & \gamma_{i} = 1, 1 \leq i < n, \gamma_{n } = 2&~\Psi(n)=\{n\},\Psi(i)=\{i,2n-i\}, 1\leq i<n\\       
\hline
 B_{n} & D_{ n+1} & \gamma_{i} = 2, 1 \leq i < n, \gamma_{n} = 1&~\Psi(n)=\{n,n+1\},\Psi(i)=\{i\}, 1\leq i<n\\
\hline
 F_4 & E_{6} & \gamma_{1} = \gamma_{2} =2, \gamma_3 = \gamma_4 = 1 &~\Psi(1)=\{2\}, \Psi(2)=\{4\},\Psi(3)=\{3,5\},\Psi(4)=\{1,6\}\\
\hline
 G_{2} & D_{4} & \gamma_{1} = 1, \gamma_{2} = 3 & \Psi(1)=\{1,3,4\},\Psi(2)=\{2\}\\
\hline
B_{n} & C_{n} & \gamma_{i} = 2, 1 \leq i < n,  \gamma_{n} = 1  &~\Psi(i)=\{i\}, 1\leq i\leq n\\
\hline
 C_{n} &  B_{n} &  \gamma_{i} = 1, 1 \leq i < n,  \gamma_{n} = 2 & ~\Psi(i)=\{i\},1\leq i\leq n\\
\hline
 B_{n} &  A_{2n-1} & \gamma_{i} = 1, 1 \leq i \leq n & ~\Psi(n)=\{n\},\Psi(i)=\{i,2n-i\}, 1\leq i<n\\
  \hline
 C_{n} &  D_{n+1} & \gamma_{i} =1, 1 \leq i \leq n&~\Psi(n)=\{n,n+1\},\Psi(i)=\{i\}, 1\leq i<n\\
 \end{array}
\]
\caption{The cases when $X = B_{n}, Y = C_{n}$, $X = C_{n}, Y = B_{n}$, and $X=B_n, C_n, Y=A_{2n-1}$ were considered in \cite{fujita18, pappe2023promotion}, \cite{fujita18}, and \cite{Kashsimilar96,  baker2000zero}, respectively. The rest appear in \cite{Kashsimilar96, schiscrivirtual15, 
bump2017crystal}. 
}\label{fig:gamma}
\end{table}

\begin{remark}
    Kashiwara \cite{Kashsimilar96} allows coefficients not necessarily equal attached to the elements in the same orbit.  The Kashiwara  procedure for the folding $X=G_2\hookrightarrow Y=A_5$is  defined by the assigment $\omega_1^X\rightarrow \omega_1^Y+2\omega_3^Y+\omega_5^Y$ and $\omega_2^X\rightarrow \omega_2^Y+\omega_4^Y $.
\end{remark}

\begin{definition}
\label{def:virtual crystal} Suppose $X$ and $Y$ are Dynkin diagrams with an embedding $\psi:X\hookrightarrow Y$ as above. 
Let $(\tilde{\B};\tilde{e}_j, \tilde{f}_j, \tilde{\varphi}_j, \tilde{\varepsilon}_j)_{j \in I^Y}$ be a normal $\mathfrak{g}_{Y}$-crystal.
A \newword{virtual $\g_X$-crystal} is a subset  $\mathcal{V} \subset \tilde{\B}$ such that $\mathcal{V}$ has a normal $\g_X$-crystal structure where for any $i \in I^X$ the crystal operators  are given by:
\begin{align}
\label{axiomvirtual1}
e^{\bf v}_{i} :=\underset{j \in \psi(i)}{\prod} \tilde{e}^{\gamma_{i}}_{j}, \qquad\qquad
f^{\bf v}_{i} := \underset{j \in \psi(i)}{\prod} \tilde{f}^{\gamma_{i}}_{j}, 
\end{align}
and for any choice of $j \in \psi(i)$, the string operators defined as: 
\begin{align}
\label{axiomvirtual2}\varepsilon_{i}:= \gamma^{-1}_{i}\tilde{\varepsilon}_{j} &\quad  \varphi_{i}:= \gamma^{-1}_{i}\tilde{\varphi}_{j}. 
\end{align}

Additionally, if a $\g_{X}$-crystal $\B$ is isomorphic to a virtual $\g_X$-crystal $\V \subset \tilde{\B}$, we call the associated isomorphism $\vt_\psi: \B \rightarrow \V$ the \newword{virtualization} map. 
\end{definition}

A priori, it is not clear that a virtual crystal is well-defined; hence a few important observations are in order. 
\begin{enumerate}
\item As noted in \cite[Rem. 5.2]{bump2017crystal} the elements $j,j' \in \psi(i)$ are not connected in $Y$, hence the associated operators $\tilde{f}_j, \tilde{f}_{j'}$ commute, so that their order in \eqref{axiomvirtual1} does not matter. Thus, $ e^{\bf v}_i$ and $ f^{\bf v}_i$ are well-defined. 
\item For any $b \in \tilde{\B}$ and $i \in I^X$, the string operators $\tilde{\varepsilon}_j(b)$ are constant with value a multiple of $\gamma_i$ over all $j \in \psi(i)$. Thus, the string operators in \eqref{axiomvirtual2} are independent of the choice of $j \in \psi(i)$ and thus well-defined. 
\item A proof that the operators $(e_i^{\bf v}, f_i^{\bf v}, \varepsilon_i, \varphi_i)$ in \eqref{axiomvirtual1} and \eqref{axiomvirtual2} endow $\V$ with the structure of a normal $\g_X$-crystal is given in \cite[Prop. 5.4]{bump2017crystal}. 
\item For any $b \in \V$, the weight map for $\V$ as a virtual $\g_X$-crystal is given by:
\[\mathrm{wt}(b)=\sum_{i\in I^X}(\varphi_i(b)-\varepsilon_i(b))\omega_i^X\]
and satisfies the property that $\psi(\wt(b))=\widetilde{\wt}(b)$ where $\widetilde{\wt}(b)=\sum_{i\in I^X}(\tilde\varphi_j(b)-\tilde\varepsilon_j(b))\omega_i^Y$ where
$\widetilde{\wt}$ is the weight map of $\tilde \B$ \cite[Rem. 5.3]{bump2017crystal}.
\end{enumerate}

\begin{remark}
\label{rem:virtual hw} By \eqref{axiomvirtual2} any virtualization maps the highest weight vector of $\B$ to the highest weight vector in $\V$, which coincides with that of $\tilde{\B}$ \cite[Proposition 5.7]{bump2017crystal}. Consequently, for each choice of embedding $\psi: X \hookrightarrow Y$ the associated virtualization is uniquely determined. 
\end{remark}

\begin{proposition}\label{virtualproperties}\cite[Proposition 6.4]{virtualtensor}  
Virtualization is closed under tensor products and direct sums. That is, for any pair of virtual crystals $(\V_1,  \tilde \B_1)$  and  $(\V_2, \tilde \B_2)$ then both $(\V_1 \otimes \V_2,  \tilde \B_1\otimes \tilde\B_2)$ and $(\V_1 \oplus \V_2,  \tilde \B_1\oplus \tilde\B_2)$ are also virtual crystals.  
\end{proposition}

\subsection{Weyl Group Orbits} 
\label{subsec:weyl-group-orbits}
Suppose $\B(\lambda) \in \g$-crystals with $\lambda \in P^+$. 
There is a natural action of the Weyl group $W$ on the set of weights $P$, determined by the action of the simple reflections $s_i$ on any $\mu \in P$:
\[
s_i(\mu) := \mu - \langle \alpha_i^\vee, \mu \rangle \alpha_i.
\]
Given $\lambda \in P^+$, we call the weights $\sigma \lambda$ in the $W$-orbit of $\lambda$ the \newword{extremal weights} of $\B(\lambda)$. 

The $W$-action on $P$, in turn, induces an action of $W$ on $\B(\lambda)$ given by flipping an element $b \in \B(\lambda)$ across the associated $i$-string; more precisely, for each $i \in I$:
\begin{equation}
 s_i(b)  : = 
    \begin{cases*}
      f^{\varphi_i(b)-\varepsilon_i(b)}_i(b)  & \;if $\varphi_i(b) - \varepsilon_i(b) \ge 0$, \\
      e^{\varepsilon_i(b)-\varphi_i(b)}_i(b)       & \;if  $\varphi_i(b) - \varepsilon_i(b)  < 0.$ \end{cases*}
\end{equation}
Setting $b_{\sigma \lambda}:= \sigma(b_\lambda)$ for any $\sigma \in W$, the \newword{W-orbit} of $b_\lambda$ is the set 
\begin{align}
\cO(\lambda):=&\{b_{\sigma\lambda} \in \B(\lambda)\; | \; \sigma\in W\}.
 \end{align}
Let $f_j^{\mathrm{max}}(b):= f_i^{\varphi_i(b)}(b)$ for any $b \in \B$, and for any $w \in W$ with a reduced expression $\mathrm{rex}(w)=s_{i_1}\dots s_{i_k}$, let
\begin{align}\label{operatorF}\F^*_w:= f_{i_1}^{\mathrm{max}}\dots f_{i_k}^{\mathrm{max}} \in \F,\end{align}
then 
\begin{align}\label{Worbit}
\cO(\lambda)=
\{f_{i_1}^\mathrm{max} f_{i_2}^\mathrm{max}\cdots f_{i_k}^\mathrm{max}(b_\lambda) \; |\; s_{i_1}s_{i_1}\dots s_{i_k} = \mathrm{rex}(w) \mbox{ with } w \leq w_0\}= \bigcup_{w \leq  w_0} \F^*_w\{b_\lambda\},
 \end{align}
where the union is taken over all $w \in W$, with $w_0$ the longest element in $W$. As before, we call the elements $b_{\sigma \lambda} \in \cO(\lambda)$ the \newword{extremal weight vectors} or \newword{keys} of $\B(\lambda)$. 

Evidently, for any $\lambda \in P^+$ and $w \in W$, if $v \in wW_\lambda$, with $W_\lambda$ the stabilizer subgroup of $\lambda$ in $W$, then $v\lambda = w\lambda$, so that $b_{v\lambda} = b_{w \lambda}$. Hence, there is a natural correspondence from $\cO(\lambda)$ and the set of minimal coset representatives of $W/W_\lambda$. 

\subsection{The Lusztig--Sch\"utzenberger Involution}
\label{sec:evac}
Given any $\lambda \in P^+$, the \newword{Lusztig--Sch\"utzenberger involution} $\xi=\xi_{\B(\lambda)}: \B(\lambda) \rightarrow \B(\lambda)$ is defined as the unique set involution such that for all $i \in I$ and $b \in \B(\lambda)$:
\[e_i\xi(b)=\xi f_{\theta(i)}(b), \quad
f_i\xi(b)=\xi e_{\theta(i)}(b),\quad \text{and} \quad \wt(\xi(b))=w_0\wt(b),
\]
where $\theta$ is the automorphism of $I$ defined by applying the longest element $w_{0} \in W$ to the simple roots: 
\[w_{0} \alpha_{i} = - \alpha_{\theta(i)}.\]

The Lusztig--Sch\"utzenberger involution exchanges the string length operators as follows
\begin{equation}
   \label{xilengths}\varepsilon_{\theta(j)}(b)=\varphi_{j}(\xi (b)), \mbox{ for $j\in I$, $b\in \B(\lambda)$.} 
\end{equation}
and acts on $\cO(\lambda)$ via 
$\xi( b_{\sigma\lambda})=b_{w_0\sigma\lambda}$ since $\mathrm{wt}(\xi( b_{\sigma\lambda}))=w_0\mathrm{wt}(b_{\sigma\lambda})
=w_0\sigma\lambda.$
Thus, given any normal $\g$-crystal $\B$ we can define $\xi_\B: \B \to \B$ by applying the appropriate $\xi_{\B(\lambda)}$ to each connected component $\B(\lambda)$ of $\B$. 

\begin{example}\label{re:schutzorbit}   
In type $A_n$ all fundamental weights are minuscule and $\theta(i)=n+1-i$. 
In types $B_n$ and $C_n$ the minuscule fundamental weights are the spin weight $\omega^{B_n}_n$ and respectively $\omega^{C_{n}}_1$ where $\theta = Id$. In particular, these fundamental weights are all extremal and form a single $W$-orbit with $\xi b_{\sigma\lambda}=b_{-\sigma\lambda}$. So then, in the type $B$ spin case, $\B(\omega_n^{B_n})=\cO(\omega_n^{B_n})$ \cite[Section 5.4]{bump2017crystal} (see also \cite{azenhas2022symplectic}).
\end{example}

The following proposition is found in \cite[Thm. 4]{torres2024virtual} in the context of Littelmann-paths and Lusztig--Sch\"utzenberger involution given by that defined in \cite{pan2018virtualization}. 

\begin{proposition}\cite[Thm. 4]{torres2024virtual}
\label{prop:virtualevac}
Given a $\g_X$-crystal $\B$ and virtualization map $\vt: \B \to \V \subset \tilde{\B}$, with $\tilde{\B}$ a $\g_Y$-crystal, we have that:
\[
\vt(\xi_\B(\B)) = \xi_{\tilde{\B}}(\vt(\B)),
\]
Thus, virtualization commutes with the Lusztig--Sch\"utzenberger involution.  
\end{proposition}

Given any $\B, \C \in \g$-crystals, it is well known that although $\B \otimes \C$ is isomorphic to $\C \otimes \B$, the isomorphism is nontrivial. In \cite{henriques2006crystals} Henriques and Kamnitzer defined this crystal isomorphism $\sigma_{\B,\C}: \B\otimes \C \to \C \otimes \B $  in terms of the Lusztig--Sch\"utzenberger involution as follows:
\begin{align}\label{eq:commutor}
\sigma_{\B ,\C}(b\otimes c) := \xi_{\C \otimes \B}(\xi_\C(c) \otimes \xi_\B(b)). 
\end{align}

\begin{proposition}\label{prop:commutor-trivial}
Let $\lambda \in P^+$ and $m$ be a positive integer. Consider any connected component $\B(\mu) \subset \B(\lambda)^{\otimes m}$ with $\mu \in P^+$. If $\B(\mu)$ appears with multiplicity one in $\B(\lambda)^{\otimes m}$, then for any $b_1 \otimes \dots \otimes b_m \in\B(\mu)$ we have that
\[
\xi_{\B(\mu)}(b_1 \otimes \dots \otimes b_m) = \xi_{\B(\lambda)}(b_m) \otimes \dots \otimes \xi_{\B(\lambda)}(b_1).
\]
\end{proposition}

\begin{proof}
We begin by noting that $\B(\lambda)^{\otimes k}\otimes \B(\lambda)^{\otimes(m-k)}$ is equal (not just isomorphic) to $\B(\lambda)^{\otimes m}$ for any $k\leq m$. 
Since $\sigma=\sigma_{\B(\lambda)^{\otimes k},\B(\lambda)^{\otimes(m-k)}}:\B(\lambda)^{\otimes m} \to \B(\lambda)^{\otimes m}$ is a crystal morphism, it follows from Schur's lemma that $\sigma$ preserves isotypic components. Thus, if $\B(\mu)$ appears with multiplicity one in $\B(\lambda)^{\otimes m}$, it follows $\B(\mu)$ is mapped isomorphically onto itself under $\sigma$. Thus, for any $b_1 \otimes \dots \otimes b_m \in \B(\mu)$ we have:
\begin{equation}\label{eq:A}
b_1\otimes \dots \otimes b_m = \sigma(b_1\otimes \dots \otimes b_m) = \xi(\xi(b_{k+1}\otimes \dots \otimes b_{m})\otimes \xi(b_{1}\otimes \dots \otimes b_{k}))
\end{equation}
Hence, $\xi(b_1\otimes \dots \otimes b_m) = \xi(b_{k+1}\otimes \dots \otimes b_{m})\otimes \xi(b_{1}\otimes \dots \otimes b_{k})$ for any $0\leq k \leq m$. 
If we then consider the connected components $\B(\nu_1) \ni b_{k+1}\otimes \dots \otimes b_{m}$ and $\B(\nu_2) \ni b_{1}\otimes \dots \otimes b_{k}$, notice that if either appeared with multiplicity higher than one in $\B(\lambda)^{\otimes k}$ and $\B(\lambda)^{\otimes (m-k)}$ respectively, this would in turn imply that $\B(\mu)$ had multiplicity higher than one in $\B(\lambda)^{\otimes m}$, a contradiction. Thus, recursively applying \eqref{eq:A} we get the desired result. 
\end{proof}

\section{Demazure Keys, Demazure Atoms and Virtualization}\label{sec:demazure}
In this section we give a type-independent proof that left and right keys are preserved under virtualization. 
We refer the reader to \cite{kashdemazure,Kashsimilar96}, \cite[Chapter 8]{kashiwara002cours} and \cite[Section 5]{bump2017crystal},  for additional background information.

\subsection{Demazure Crystals}  \label{subsec:demaz}
Let $V(\lambda)$ be an integrable highest weight module with highest weight $\lambda \in P^+$ and $w \in W$. The \newword{Demazure module} $V_w(\lambda)$ is a $\fb$-module generated by the one dimensional weight space $V(\lambda)_{w\lambda}$ of weight $w \lambda$ under action of any Borel subalgebra $\fb \subset \g$ \cite{demazure74}. 
Littelmann \cite{littelmann95conj} proved in all classical types and Kashiwara \cite{kashdemazure} generalized to any symmetrizable Kac--Moody Lie algebra that $V_w(\lambda)$ admits a crystal basis  that arises as a certain subset of $\B(\lambda)$, which we now describe.

For any $w \in W$ with a reduced expression $s_{i_1}\dots s_{i_k}$ define, 
\begin{align}\label{fwew}
    \F_w:= \bigcup_{m_i \in \Z_{\geq 0}} \{f_{i_1}^{m_1}\dots f_{i_k}^{m_k} \}\subset \F \qquad \text{and} \qquad \E_w:= \bigcup_{m_i \in \Z_{\geq 0}} \{e_{i_1}^{m_1}\dots e_{i_k}^{m_k} \}\subset \E,
\end{align}
so that $\F_{w_0}=\F$, $\E_{w_0}=\E$ and $\F_e=\E_e=\mathbbm{1}$ with $e$ the identity of $W$.

\begin{definition}
Given $\lambda \in P^+$ and $w \in W$, the \newword{Demazure crystal} $\B_w(\lambda)$ is given by, 
\[
\B_w(\lambda):= \F_w\{b_\lambda\} \subset \B(\lambda).
\]
Similarly, we define the \newword{opposite Demazure crystal} $\B^w(\lambda)$ as
\[
\B^w(\lambda):= \E_{ww_0}\{b_{w_0\lambda}\} \subset \B(\lambda).
\]
\end{definition}
In particular, we have $\F_w^*(b_\lambda) = b_{w\lambda}$ and $\displaystyle\bigcup_{v \leq  w} \F^*_v\{b_\lambda\}\subseteq \B_w(\lambda)$ with $\B_{w_0}(\lambda) = \B(\lambda)$ and $\B_e(\lambda) = \{b_\lambda\}$. Hence, Demazure crystals can be seen as certain subsets of $\B(\lambda)$ with lowest weight vector $b_{w\lambda}$ that interpolate between the highest weight $b_\lambda$ and the complete irreducible crystal $\B(\lambda)$. By a lowest weight vector $b$ in a Demazure crystal we mean a vector $b$ in that Demazure such that, for all $i\in I$, $f_i(b)=0$ or $f_i(b)$ is not in that Demazure crystal   \cite{ag21}.

Moreover, by definition $\B_w(\lambda)$ admits a filtration by Demazure  crystals, so that for any chain in $W$,
$w_1 \leq w_2 \leq \dots \leq w_k,$ we have
\begin{equation}\label{eq:Dem filtration}
\B_{w_1}(\lambda) \subseteq \B_{w_2}(\lambda)\subseteq \dots \subseteq \B_{w_k}(\lambda).
\end{equation}
Hence, any $\g$-crystal has filtration by Demazure crystals. 

In a similar fashion, Demazure crystals can be decomposed into smaller constituents called \newword{Demazure atoms}, 
\begin{align*}\mathring{\B}_{ w}(\lambda):={\B}_{w }(\lambda)\setminus\bigcup\limits_{ \begin{smallmatrix}\nu\in W\\\nu<w \end{smallmatrix}}{\B}_{\nu}(\lambda).
\end{align*}
In particular, each atom $\mathring{\B}_{ w}(\lambda)$ uniquely contains the extremal weight vector $b_{w\lambda}$.

\begin{remark}
The term Demazure atom has been used in the literature to mean both the \emph{crystal} subset $\mathring{\B}_{w}(\lambda)\subset \B(\lambda)$ and its corresponding \emph{polynomial} character $\mathcal{A}_{w\lambda} \in \Z[x]$, with the notation $\overline{\B}_{w}(\lambda)$ used by Kashiwara to denote this subset instead \cite[Ch. 9.1]{kashiwara002cours}, \cite{SarahMasonAtoms}. In this article, we will only refer to the crystal subset by this name and never discuss its character. 
\end{remark}
 
\subsection{Dilation of Crystals}\label{sec:dilation}
In this section we recall a particularly important virtualization which we will be used in the remainder of the article. We refer the reader to \cite{Kashsimilar96} and \cite[Chp. 8]{kashiwara002cours} for any proofs and details omitted herein. 

\begin{definition}\cite[Theorem 3.1]{Kashsimilar96} Given a positive integer $m$ and $\lambda\in P^+$, the \newword{$m$-dilation} of $\B(\lambda)$ is the unique embedding
\[
\D_{m}:\B(\lambda)\hookrightarrow \B(m\lambda)
\]
defined by mapping the highest weights to each other, $b_\lambda \mapsto b_{m\lambda}$,  and extending to any $b={f}_{i_{1}
}\cdots{f}_{i_{l}}(b_{\lambda}) \in \B(\lambda)$ via:
\[
{f}_{i_{1}
}\cdots{f}_{i_{l}}(b_{\lambda}) \mapsto {f}_{i_{1}}^{m}\cdots{f}_{i_{l}}^{m}(b_{m\lambda}).
\]
\end{definition}
It follows directly from the definition\footnote{The original definition introduced by Kashiwara used \eqref{dilat1} and \eqref{dilat2} since it was given in the more general context of Kac--Moody algebras.} that for any vertex $b\in \B(\lambda)$ and $i\in I$,
\begin{align}\label{dilat1}
\D_{m}(f_ib)=f_i^m\D_m(b),\;\; \D_m(e_ib)=e_i^m\D_m(b) \mbox{ and }
\end{align}
\begin{align}\label{dilat2}\varphi_i(\D_m(b))=m\varphi_i(b), \;\varepsilon_i(\D_m(b))=m\varepsilon_i(b), \;
\mathrm{wt}(\D_m(b))=m\mathrm{wt}(b).
\end{align}

Moreover, for $m,n$ positive integers $\D_{mn}$ factors through $\D_n$ and $\D_m$, 
\[
\begin{tikzcd}
& \B(m\lambda) \arrow[dr,"\D_{n}",hook] &\\
\B(\lambda) \arrow[ur, "\D_{m}",hook] \arrow[dr,"\D_{n}",hook] &&\B(nm\lambda). \\ 
& \B(n\lambda)\arrow[ur, "\D_{m}",hook] &
\end{tikzcd}
\]
Recall Proposition \ref{decomposition}. Thus, more generally, for any crystal $\B$, the $m$-dilation $\D_m$ acts by dilating each connected component individually.

\begin{proposition}{\cite[Lem. 8.1.2]{kashiwara002cours}, \cite[Corollary 2.1.3]{lecouvey2003schensted}} \label{prop:dilation tensor}
For any $\lambda, \mu \in P^+$ and positive integer $m$, the dilation map $\D_m$ preserves highest weight vectors, so that for any $u \in \B(\mu)$ with $\varphi_i(b_\lambda)\geq \varepsilon_i(u)$,
\[ b_{\lambda + \wt(u)} \mapsto b_{m\lambda+m\wt(u)}.\]
Thus $\D_m: \B(\lambda+\wt(u)) \hookrightarrow \B(m\lambda+m\wt(u))$
and consequently,
\[\D_m(\B(\lambda)\otimes \B(\mu)) \subset \D_m(\B(\lambda) )\otimes \D_m(\B(\mu)).\]
Thus, $\D_m\otimes \D_m: \B(\mu)\otimes \B(\lambda)\hookrightarrow \B(m\mu)\otimes \B(m\lambda)$ is an $m$-dilation map. More generally, $\D_m(\bigotimes_{i} \B(\lambda_i) ) \subset \bigotimes_{i} \D_m(\B(\lambda_i))$ for any family $\lambda_i \in P^+$.
\end{proposition}

Now, let $\F\{b_\lambda^{\otimes m}\} \subset \B(\lambda)^{\otimes m}$ denote the unique connected component with highest weight $m\lambda$ in $\B(\lambda)^{\otimes m}$ and let 
\[G_m: \B(m\lambda)\longrightarrow \F\{b_\lambda^{\otimes m}\}\]
be the induced crystal isomorphism mapping $b_{m\lambda} \to b_\lambda^{\otimes m}$. 
Thus, we have a canonical embedding:
\begin{equation}
\Theta_{m}:= G_m \circ \D_m:
\B({\lambda})\hookrightarrow \F\{b_\lambda^{\otimes m}\} \subset \B({\lambda
})^{\otimes m}.  \label{embded}
\end{equation} 
We will abuse terminology, and refer similarly to $\Theta_m$ as $m$-dilation. \\

The following result was proven by Kashiwara in \cite[Theorem 5.1]{Kashsimilar96} in complete generality for any symmetrizable Kac--Moody Lie algebra for the dilation map $\D_m$. Combined with the uniqueness property of virtualizations (see Remark \ref{rem:virtual hw}), we restate it as follows. 

\begin{theorem}\label{thm:dilation is virt}\cite[Theorem 5.1]{Kashsimilar96} Given any Dynkin diagram $X$ and positive integer $m$, let $\Psi: X \to X$ be the automorphism determined by the assignment $\omega_i \mapsto m \omega_i$ for each $\omega_i \in P$ and $\vt_\psi$ the associated virtualization. Then, for any $\lambda \in P^+$, we have that $\vt_\psi(\B(\lambda)) = \Theta_m(\B(\lambda))$. In particular, $m$-dilation is a virtualization. 
\end{theorem}

As we will see, dilation plays a central definition of left and right keys. To get there, we must first understand the connection between dilation and extremal weight vectors.

\begin{proposition}\cite[Prop. 8.3.2]{kashiwara002cours}\label{prop:theta-extremal}
Given positive integers $m,n$, $\lambda \in P^+$, and $w \in W$, then
\begin{enumerate}
\item for any extremal weight vector $b_{w\lambda} \in \cO(\lambda)$, we have $\Theta_m(b_{w\lambda}) = b_{w\lambda}^{\otimes m}$. In particular, $\Theta_m(\cO(\lambda))=\cO(m\lambda)$, and
\item $\Theta_{mn} = \Theta_m \circ \Theta_n =\Theta_n \circ \Theta_m$.
\end{enumerate}
\end{proposition}

\begin{corollary}\label{cor:multipledilation}
For $b\in\B(\lambda)$ if $\Theta_m(b) = b_{w\lambda}\otimes b' \otimes b_{w'\lambda}$ for some $b' \in \B(\lambda)^{\otimes (m-2)}$ and $w,w'\in W$, then $\Theta_n(b_{w\lambda}\otimes b' \otimes b_{w'\lambda})=b^{\otimes n}_{w\lambda}\otimes \Theta_n(b') \otimes b^{\otimes n}_{w'\lambda} \in \B(\lambda)^{\otimes mn}$.
Moreover, if $\Theta_{m}(b)=b_{w_{1}\lambda}\otimes\cdots\otimes
b_{w_{m}\lambda}\in\cO(\lambda)^{\otimes m}$ for a sequence $w_1,\dots,w_m\in W$, then
$\Theta_{nm}(b)=\Theta_n\Theta_m(b)=b^{\otimes n}_{w_{1}\lambda}\otimes\cdots\otimes
b^{\otimes n}_{w_{m}\lambda}$ determines the same sequence $W$ up to multiplicity. 
\end{corollary}

It was also shown in \cite[Prop. 8.3.2]{kashiwara002cours} that for any $b\in \B(\lambda)$ if $m$ is adequate then there exist $w_i \in W$
satisfying $w_1\ge \cdots \ge w_m$ in $W$ such that 
\begin{equation}\label{eq:theta-extremal product}
\Theta_m(b) = b_{w_1\lambda} \otimes b_{w_2\lambda} \otimes \dots \otimes b_{w_m\lambda} \in \cO(\lambda)^{\otimes m}.
\end{equation}
This proof, however, is inductive and does not provide explicit bounds for how large such an $m$ must be. In the following  theorem, we expand Proposition \ref{prop:theta-extremal}(1) for any $b \in \B(\lambda)$ by providing sufficient and necessary conditions for $m$ that ensure the decomposition 
$\Theta_m(b) = b_{w\lambda}\otimes b' \otimes b_{w'\lambda}$ with $b_{w\lambda},b_{w'\lambda}$ extremal and $b' \in \B(\lambda)^{\otimes (m-2)}$ occurs. 

\begin{theorem}\label{thm:extremal endpoints}
Let $m \in \mathbb{N}$. 
For all $b \in \B(\lambda)$, there exists $b' \in \B(\lambda)^{\otimes (m-2)}$ and fixed $w\geq w' \in W$ such that 
\[\label{ell}\Theta_m(b) = b_{w\lambda}\otimes b' \otimes b_{w'\lambda} \quad \text{ if and only if } \quad m \geq \ell\]
where $\ell=\max \{ \text{length}(\rho) \; | \; \rho \text{ is an i-string in } \B(\lambda)\text{ for some } i \in I\}.$
\end{theorem}

\begin{proof}
Fix $i \in I$. Let $m$ be given and suppose $\rho=( b_0 \overset{i}{\to} b_1 \overset{i}{\to}  \dots \overset{i}{\to} b_k) \subset \B(\lambda)$ is any $i$-string  of length $k$. Now, for any $i$-string there exists a filtration by Demazure crystals of $\B(\lambda)$ (see \eqref{eq:Dem filtration}) such that $\{ b_\lambda \}  \subset \B_{ w}(\lambda) \subset \B_{s_i w}(\lambda) \subset \B(\lambda)$ for some $s_i w > w \in W$. In particular, $w$ may be chosen such that the $i$-string connecting the Demazure lowest weight $b_{s_iw} \in \B_{s_iw}(\lambda)$ to $\B_w(\lambda)$ is the $i$-string of maximal length in $\B(\lambda)$. Since $b_{s_i w}$ is extremal, it follows that $e_i^{\varepsilon_i(b_{s_i w})}(b_{s_i w}) = b_w$ is also extremal. Hence, if $b_0$ in $\rho=( b_0 \overset{i}{\to} b_1 \overset{i}{\to}  \dots \overset{i}{\to} b_k) $ is not extremal then $\rho$ can be replaced with another $i$-string $\rho'$ of length $ k'\ge k$ such that $b'_{0} \in \cO(\lambda)$. Thus, without loss of generality we can assume that $b_0$ is an extremal weight vector. 

Now, since $b_0 \in \cO(\lambda)$, then so is $b_k$. Hence by Proposition \ref{prop:theta-extremal} $\Theta_m(\cO(\lambda)) \subset \cO(m\lambda)$, implies that $\Theta_m(\rho) \subseteq \rho \otimes \Theta_{m-1}(\rho) \subseteq \B(\lambda)^{\otimes m}$. So then by  (\ref{formula:nfis}) below 
\begin{align}
\label{formula:nfis}
f_{i}^{n}(b_1 \otimes b_2)  =
\begin{cases}
  f_{i}^{n}(b_1)\otimes b_2  & \text{ if } \varphi_{i}(b_1) \geq \varepsilon_{i}(b_2) + n, \\
  f_{i}^{\varphi_{i}(b_1) -\varepsilon_{i}(b_2)}(b_1) \otimes f_{i}^{n -\varphi_{i}(b_1)+\varepsilon_{i}(b_2)}(b_2) & \text{ if }  \varepsilon_{i}(b_2) \leq \varphi_{i}(b_1) \leq \varepsilon_{i}(b_2) + n, \\
  b_{1} \otimes f_{i}^{n}(b_2) & \text{ if }
 \varepsilon_{i}(b_2) \geq \varphi_{i}(b_1)
\end{cases}
\end{align} one has
\begin{equation}\label{eq:B}
f_i^r\Theta_m(b_0) = f_i^r(b_0 \otimes b_0^{\otimes m-1}) = \begin{cases}
f_i^r(b_0)\otimes b_0^{\otimes m-1} &; r < k,\\
b_k \otimes f_i^{r-k}(b_0^{\otimes m-1}) &; r\geq k.
\end{cases}
\end{equation}
Thus, if $m<k\leq \ell$ then $\Theta_m(b_1) = f_i^m(b_0) \otimes b_0^{\otimes m-1}$ with $f_i^m(b_0) \not\in \cO(\lambda)$. 

Conversely, if $m\geq \ell \geq k$ and $b \in \B(\lambda)$ lies in an $i$-string with $b_0$ extremal, then it follows that  $\Theta_m(b_i) = b_k \otimes f_i^{r-k}(b_1^{\otimes m-1}) $ for all $1\leq i \leq k$ with $b_k \in  \cO(\lambda)$. If however, $b$ does not lie on an $i$-string with $b_0$ extremal, let $b'_0 \in \cO(\lambda)$ be such that $b=f_{i_1}^{n_1} \dots f_{i_s}^{n_s}(b'_0)$. Moreover, since $m \geq \ell$ then by \eqref{eq:B} there exists some subset $\{j_\alpha\} \subset \{i_\beta\}$ such that
\[
\Theta_m(b) = f_{i_1}^{m(n_1)} \dots f_{i_s}^{m(n_s)}({b'_0}^{\otimes m}) = f_{j_1}^{\max} \dots f_{j_t}^{\max}(b'_0) \otimes b'' \in \cO(\lambda) \otimes \B(\lambda)^{\otimes (m-1)}. 
\]
Setting $b_{w_1\lambda}= f_{j_1}^{\max} \dots f_{j_t}^{\max}(b'_0)$ the first claim follows. 

The proof that any $b \in \B(\lambda)$ satisfies $b = b' \otimes b_{w_2\lambda}$ if and only $m \geq \ell$ follows similarly to the first part by considering $\Theta_m(b) \in \B(\lambda)^{\otimes (m-1)} \otimes \B(\lambda)$ and noting that for any $i$-string with $b_k$ extremal we have:
\begin{equation}
e_i^r\Theta_m(b_k) = e_i^r( b_k^{\otimes m-1} \otimes b_k ) = \begin{cases}
 b_k^{\otimes m-1} \otimes e_i^r(b_k) &, r < k,\\
 e_i^{r-k}(b_k^{\otimes m-1}) \otimes b_0 &, r\geq k.
\end{cases}
\end{equation}

\end{proof}
\begin{example} Recalling Example \ref{re:schutzorbit} where 
$\B(\omega_i^{A_n})=\cO(\omega_i^{A_n})$, $\B(\omega_n^{B_n})=\cO(\omega_n^{B_n})$, $\B(\omega_1^{C_n})=\cO(\omega_1^{C_n})$, we see that in each of these cases, all the vertices are extremal weight vectors, and hence $\ell=1$. When $\omega \in \{\omega_1^{B_n},\omega_2^{B_n},\dots,2\omega_n^{B_n} \}$ or $\omega \in\{\omega_2^{C_n}\dots,\omega_n^{C_n} \}$, then the crystal $\B(\omega)$ consisting of the non-spin admissible columns in either type $B_n$ or alternatively in type $C_n$ of length at least two, then we have $\ell =2$. In particular, this implies that Theorem \ref{thm:extremal endpoints}  is an extension of \cite[Proposition 2.9]{lenlub2015} and \cite[ Proposition 3.1.9]{lecouvey2003schensted}.
    \end{example}

\subsection{Keys and Demazure Atoms} Recall that by \cite[Prop. 8.3.2]{kashiwara002cours} and Theorem \ref{thm:extremal endpoints} for any $b \in \B(\lambda)$ we can always find an $m$ such that $\Theta_m(b) = b_{w\lambda}\otimes b' \otimes b_{w'\lambda}$ for some $w\geq w' \in W$ and $b' \in \B(\lambda)^{\otimes (m-2)}$. 
    
\begin{definition}
\label{DefKeys} For $b\in \B(\lambda)$ and any $m$ such that $\Theta_m(b) = b_{w\lambda}\otimes b' \otimes b_{w'\lambda}$, the \newword{right key} $K^{+}(b)$ and \newword{left key} $K^{-}(b)$
of $b$ are defined to be the unique extremal weight vectors,  
\begin{align}\label{eq:leftright}
K^{+}(b)=b_{w\lambda}\text{ and }K^{-}(b)=b_{w'\lambda}.
\end{align}
\end{definition}

Evidently, it immediately follows from Proposition \ref{prop:theta-extremal} that $K^+(b)=K^-(b)$ if and only if $b\in \cO(\lambda)$, with  $K^{+}(b_{w\lambda})=K^{-}(b_{w\lambda})=b_{w\lambda}$ for any $w\in W$. 

\begin{remark}
     In Lascoux's original definition {\cite{LaSchu90keys}}, based on the tableau model, the order of the tensor product is reversed, with the right (resp. left) key appearing as the rightmost (resp. leftmost) factor in $\Theta_m(b)$. 
 Nonetheless, despite the discrepancy with our conventions, we chose to keep Lascoux's terminology in order to stay consistent with the literature.
\end{remark}
 \begin{definition}\label{def:demazatom}{ \cite{kashiwara002cours, LaSchu90keys} }
Given $\lambda\in P^+$ and $w\in W$, define the \newword{Demazure atom} $\mathring{\B}_{w}(\lambda)$ as the subset, 
\begin{align*}
\mathring {\B}_{w}(\lambda):&=\{b\in \B(\lambda): K^+(b)=
b_{w\lambda}\}   \subset \B_w(\lambda)
\end{align*}
and the \newword{opposite Demazure atom} $\mathring{\B}^{w}(\lambda)$ as the subset,
\begin{align*}
\mathring {\B}^{w}(\lambda):&=\{b\in \B(\lambda): K^-(b)=b_{w\lambda}\}  \subset \B^w(\lambda).
\end{align*}
 \end{definition}
 Thus, the right (resp. left) key assigns to each vertex of a crystal the extremal weight vector that generates the smallest possible (resp. opposite) Demazure crystal containing the given vertex. The Demazure atoms consist of all such vectors for whom this minimal Demazure crystal is the same. 
 
It is important to note that generally Demazure atoms and their opposites do not have a Demazure crystal structure. Nonetheless, it is classical fact \cite{kashiwara002cours,LaSchu90keys,sa21a},  and evident from Definition \ref{DefKeys}, that (opposite) Demazure crystals can be built from (opposite) atoms, 
\[
\B_w(\lambda) = \bigsqcup_{\nu \leq w} \mathring\B_\nu(\lambda)=\{b\in B(\lambda): K^+(b)=b_{\nu\lambda},\; \nu\le w \}.
\] 
\begin{remark}
Originally, Lascoux-Sch\"utzenberger \cite{LaSchu90keys} referred to atoms as the \textit{standard basis}.
Keys  in type $A_{n-1}$ have their origin in the $GL(n,\mathbb{C})$ \emph{standard} bases to detect the semistandard tableaux which are \emph{standard} (in the sense of \text{standard monomial theory}) on a Schubert variety. In each Demazure crystal atom there exists exactly one key tableau and the right key map detects the Demazure crystal atom that contains a given semistandard Young tableau \cite[Theorem 3.8]{LaSchu90keys}.
\end{remark}

Finally, although the following theorem is known to experts via the isomorphism between the path model of Lakshmibai--Seshadri and Kashiwara's crystal theory for irreducible highest weight modules \cite{littelmann1995paths, Kashsimilar96,kashiwara002cours}, in the following we provide a new elementary and model independent argument. 

\begin{theorem}\label{thm:Keys-Schutz-commute} 
For any $\lambda \in P^+$ and $b \in \B(\lambda)$, 
\[K^+(\xi (b))=\xi K^-(b).\]
\end{theorem}

\begin{proof}
Suppose $b \in \B(\lambda)$ and $m$ is such that $\Theta_m(b) = b_{w_1\lambda} \otimes \dots \otimes b_{w_m\lambda} $ for some $w_i \in W$, so that $K^+(b) = b_{w_1\lambda}$. By Theorem \ref{thm:dilation is virt} any $m$-dilation is a virtualization, hence by Proposition \ref{prop:virtualevac} it follows that $\Theta_m \xi = \xi \Theta_m$. So then, applying Proposition \ref{prop:commutor-trivial}  to the connected component $\F\{b_\lambda^{\otimes m}\}= \B(m\lambda)$, from Corollary \ref{cor:multipledilation} we have that:
\begin{align}\Theta_m\xi_{\B(\lambda)}(b)=&
\xi_{\B(m\lambda)}\Theta_m(b) 
= \xi_{\B(m\lambda)}(b_{w_1\lambda} \otimes \dots \otimes b_{w_m\lambda} ) = \xi_{\B(\lambda)} (b_{w_m\lambda}) \otimes \dots \otimes\xi_{\B(\lambda)} (b_{w_1\lambda})\nonumber
\\
=&b_{w_0w_{m}\lambda}\otimes \cdots\otimes
 b_{w_0w_{1} \lambda}.\label{schutzLS}\end{align}
Thus, $K^+(\xi(b))= \xi b_{w_m\lambda} = b_{w_0 w_m\lambda}=\xi K^-(b)$. 
\end{proof}

As an immediate consequence we also obtain the equalities, 
  \begin{align}
\xi{\mathring{\B}}_{w}(\lambda)= \mathring{{\B}}^{w_0w}(\lambda) \qquad \text{and} \qquad \B^{w_0w}(\lambda)=\xi \B_w(\lambda). \label{opBoverlinatom}
\end{align} 

\begin{remark} \label{re:schutzleftrightkeys}
The sequences in  $W$ \eqref{schutzLS} produced by Theorem \ref{thm:Keys-Schutz-commute} are Lakshmibai--Seshadri (L-S) paths (see also \cite[Lemma 3.1 (b)]{littelmann1994LS}). Indeed Theorem \ref{thm:Keys-Schutz-commute} induces an action of the Lusztig--Sch\"utzenberger involution on the crystal of L-S paths ${\mathbf B}(\lambda)$.  More precisely, $\xi(\tau;\mathbf{a})=\xi(\tau_0>\cdots>\tau_r; 0<a_1<\cdots<a_r<1)=(w_0\tau_r>\cdots>w_0\tau_0;0<1-a_r<\cdots<1-a_1<1)$ where $(\tau;\mathbf{a})$ is an L-S path of ${\mathbf B}(\lambda)$.
Within the Lenart--Postnikov alcove path model \cite{lenart2008combinatorial},  \cite[Definition 5.2, Remark 5.3, Corollary 6.2]{lenart2007combinatorics} the initial key and the final key are interchanged by the Lusztig--Sch\"utzenberger involution. 
\end{remark}

\subsection{Virtual Keys}
\label{virtualkeys}
 We are now ready to prove that virtualization preserves left and right keys. These same results were proven by Azenhas--Santos \cite{azenhas2024virtualkeys} using Baker's virtualization from types $C_n$ into $A_{2n-1}$ \cite{baker2000zero}. Here we generalize these results to any virtualization map associated to a Dynkin diagram embedding $\psi: X \hookrightarrow Y$. 

\begin{proposition}\label{thm:virt-dilation-commute}
For any virtualization map $\vt_\psi:\B(\lambda) \to \V \subset \tilde{\B}(\psi(\lambda))$ and any positive integer $m$, the induced map $\vt_\psi^{\otimes m}: \Theta_m(\B(\lambda)) \to  \Theta_m(\V)$ is a virtualization map. Consequently, $\Theta_m \vt_\psi = \vt_\psi^{\otimes m} \Theta_m$.
\end{proposition}

\begin{proof}
    We begin by noting that by \cite[Thm. 5.8]{bump2017crystal} for any Dynkin diagram embedding $\psi: X \hookrightarrow Y$ with associated virtualization maps $\vt_\psi: \B \to \V \subset \tilde{\B}$ and $\vt'_\psi: \C \to \mathcal{W} \subset \tilde{\C}$, the map $\vt \otimes \vt': \B \otimes \C \to \V \otimes \mathcal{W}$ is a virtualization map. Thus, it follows that $\vt_\psi^{\otimes m}: \B(\lambda)^{\otimes m} \to  \tilde{\B}(\psi(\lambda))^{\otimes m}$ is indeed a virtualization.

    Now, by Theorem \ref{thm:dilation is virt} we know that $\Theta_m$ is a virtualization on the same Lie type. Since composition of virtualizations is again a virtualization, it follows that both $\Theta_m \vt_\psi: \B(\lambda) \to \F \{ b_{\psi(\lambda)}^{\otimes m}\}$ and $\vt_\psi^{\otimes m} \Theta_m: \B(\lambda) \to \F \{ b_{\psi(\lambda)}^{\otimes m}\}$ are also virtualizations. Moreover, since the choice of Dynkin diagram embedding uniquely determines the virtualization (see Remark \ref{rem:virtual hw}) it follows that $\Theta_m \vt_\psi = \vt_\psi^{\otimes m} \Theta_m$.
\end{proof}

Recall that for any Dynkin diagram embedding $\psi: X \to Y$ we have an induced Weyl group embedding 
\[\psi: W^X \to \widetilde{W}^X:= \langle{s}_i := \prod_{j \in \psi(i)} \tilde{s}_i \; |\; i \in I^X \rangle \subset W^Y = \langle \tilde{s}_j\; | \; j \in I^Y \rangle.\]
Denote by $\cO^X(\lambda)$ the $W^X$-orbit of $b_\lambda$ and by $\widetilde{\cO}^X(\psi(\lambda))$ the $\widetilde{W}^X$-orbit of $b_{\psi(\lambda)}$.

\begin{lemma} \label{lem:orbit-virt}
For any Dynkin diagram embedding $\psi: X \to Y$ with virtualization map $\vt_\psi: \B(\lambda) \to \V \subset \tilde{\B}(\psi(\lambda))$, we have that
\[
\vt_\psi \cO^X(\lambda) = \widetilde{\cO}^X(\psi(\lambda)):= \{ b_{w \psi(\lambda)} \; |\; w \in \widetilde{W}^X\} \subset \cO^Y(\psi(\lambda)).
\]
\end{lemma}

\begin{proof}
Suppose $b_{w\lambda} \in \cO^X(\lambda)$ for some $w \in W^X$. Then for any reduced expression $s_{i_1}\dots s_{i_k}$ of $w$, we have $f_{i_1}^{\max}\dots f_{i_k}^{\max} (b_\lambda) = b_{w\lambda}$. Hence, 
\[
\vt_\psi(b_{w\lambda}) = (f_{i_1}^{\bf v})^{\max}\dots (f_{i_k}^{\bf v})^{\max} (b_{\psi(\lambda)}) = \prod_{j \in \psi(i_1)} \tilde{f}_j^{\max} \dots  \prod_{j \in \psi(i_k)} \tilde{f}_j^{\max}(b_{\psi(\lambda)})= b_{\psi(w)\psi(\lambda)}, 
\]
where the last equality holds since, by definition, $s_i= \prod_{j \in \psi(i)} \tilde{s}_i \in \tilde{W}^X \subset W^Y$,  in which case $s_{i_1}\dots s_{i_k}$ is a reduced expression for $\psi(w)$ in $\tilde{\cO}^X(\psi(\lambda))$. 
\end{proof}

 We now conclude that a virtualization map preserves left and right keys and thereby embeds type $X$ Demazure crystals and atoms into those of type $Y$. 

\begin{theorem} \label{thm:keys-virt-commute} 
Let $\psi:X \to Y$ be a Dynkin diagram embedding with virtualization map $\vt=\vt_\psi:\B(\lambda) \to \tilde{\B}(\psi(\lambda))$. Then, for any $b \in \B(\lambda)$: 
\[\vt(K^+(b))=K^+(\vt(b)) \text{ and } \vt(K^-(b))=K^-(\vt(b)).\] 
Thus, virtualization embeds Demazure crystals and atoms correspondingly, so that for any $w \in W^X$ we have
\[ \B_w(\lambda) \overset{\vt}{\hookrightarrow} \tilde{\B}_{\psi(w)}(\psi(\lambda)) \qquad \text{ and } \qquad \mathring {\B}_{w}(\lambda)\overset{\vt}{\hookrightarrow} \mathring{\tilde{\B}}_{\psi(w)}(\psi(\lambda))\]
and similarly for their opposites. 
\end{theorem}

\begin{proof}
Let $b \in \B(\lambda)$ and $m$ be such that $\Theta_m(b) = b_{w_1\lambda} \otimes \dots \otimes b_{w_m\lambda} \in \cO(\lambda)^{\otimes m}$ with $w_1\ge \cdots \ge w_m$. Then by Proposition \ref{thm:virt-dilation-commute} and Lemma \ref{lem:orbit-virt}, 
\[
\Theta_m \vt(b) = \vt^{\otimes m} \Theta_m(b) = \vt(b_{w_1\lambda}) \otimes \dots \otimes \vt(b_{w_m\lambda})\in \tilde{\cO}^X(\psi(\lambda))^{\otimes m}
\]
 where $\psi(w_1)\ge\cdots\ge\psi(w_m)$.
Thus, $K^+(\vt(b)) = \vt(b_{w_1 \lambda}) = \vt K^+(b)$ and $K^-(\vt(b)) = \vt(b_{w_m \lambda}) = \vt K^-(b)$ as desired.
\end{proof}

\subsection{Computational Consequences of Our Results} 
Lakshmibai--Littelmann \cite{Lakshlittelmann2002LS} identifies the standard monomials of Richardson varieties  with certain L--S paths. This allows an identification
of a standard monomial of degree $m$ of a Richardson variety
with an ordered sequence of $m$ elements in the 
corresponding crystal, satisfying certain
conditions on their left and right keys. This 
condition, from a crystal-theoretic point of view,
can be obtained by characterizing when
the tensor of $m$ elements in $\B_v^w(\lambda):=\B_v(\lambda)\cap \B^w(\lambda)$, $ w\le v$, 
is an element  in  $\B_v^w(m\lambda)$. Furthermore,
since keys are not manifest in every model
of crystals and dilation is not an effective
way to compute keys, one can use virtualization to reduce the computation to simply-laced root systems.

Let $\psi:X \to Y$ be a Dynkin diagram embedding with virtualization map $\vt=\vt_\psi:\B(\lambda) \to \tilde{\B}(\psi(\lambda))$. Assume that one has a combinatorial model for $\B(\lambda)$, $\tilde{\B}(\psi(\lambda))$ and that we have an algorithm for computing $\vt$ and its inverse 
\[\vt^{-1}: \vt(\mathcal{B}(\lambda)) \rightarrow \mathcal{B}(\lambda).\] 

\noindent Additionally, assume that the model for $\tilde{\B}(\psi(\lambda))$ includes algorithms for computing $\xi(b)$ as well as $K^{+}(b)$ and $K^{-}(b)$ for any $b \in \tilde{\B}(\psi(\lambda))$. Then  Proposition \ref{prop:virtualevac}  and Theorem  \ref{thm:keys-virt-commute} imply that, to compute $\xi(b)$ as well as $K^{+}(b)$ and $K^{-}(b)$ for any $b \in \mathcal{B}(\lambda)$, it suffices to compute the virtualization map $\vt$, apply the existing corresponding  algorithm, and apply the inverse map $\vt^{-1}$.

In particular, when $\psi$ is one of the Dynkin diagram embeddings   $ C_n\hookrightarrow A_{2n-1}$ or $B_n\hookrightarrow C_n \hookrightarrow A_{2n-1}$ then, in the tableau models,  $\vt^{-1}$ is explicit and described by the reverse Schensted insertion \cite{   baker2000zero,  azenhas2022symplectic, azenhas2024virtualkeys} and by the un-splitting operation in the case of $B_n\hookrightarrow C_n$. In the next section we explore the latter case in depth.

\section{Applications to Orthogonal Kashiwara--Nakashima Tableaux}
\label{applicationsb}
In this section we follow  \cite{baker00bn,kashiwaranakashima1994crystal,lecouvey2003schensted} and  \cite{hong2002introduction}.
Let $\g = \mathfrak{so}_{2n+1}$ with fundamental weights $\omega^{B_n}_i$ given by  
\[
\omega^{B_n}_i=\begin{cases} (1^i, 0^{n-i}) & i \neq n\\
\frac{1}{2} (1^n) & i= n \end{cases}
\]
so that for any $\lambda \in P^+$ we can write $\lambda = \sum_{i=1}^n{a_i\omega^{B_n}_i}$ for some choice of coefficients $a_i$. 

Hence, we have a decomposition for each part $\lambda_j$ of $\lambda$,
\[
\lambda_j = \sum_{i=j}^{n-1} a_i + \frac{1}{2}a_n= \sum_{i=j}^{n-1} a_i + \frac{1}{2}(a_n-\delta) + \delta \frac{1}{2}
\]
with $\delta = 1$ or $0$ depending on the parity of $a_n$. Consequently, to every such $\lambda$ we can associate a diagram obtained by concatenating a special ``half width'' or spin column of height $n$ and a Young diagram of shape $\mu$ whose $i$-th part is given by $\mu_i:= \lambda_i - \delta \frac{1}{2} \in \mathbb{N}$. It follows then that any highest weight representation $V(\lambda)$ arises as a summand of tensor products of the \emph{standard representation} $V(\omega^{B_n}_1)$ and the \emph{spin representation} $V(\omega_n^{B_n})$.  

For $\g = \mathfrak{sp}_{2n}$ we have $\omega^{C_n}_i=(1^i, 0^{n-i}) $ for all $1 \leq i \leq n$ and, as above, write $\lambda = \sum_{i=1}^n{a_i\omega^{C_n}_i}$ for any $\lambda \in P^{+}$. In this case, dominant integral weights $\lambda$ coincide with partitions of at most $n$ parts, and any highest weight representation $V(\lambda)$ arises as a summand of tensor products of the standard representation $V(\omega_1)$.
In order to combinatorially model the representations of $\mathfrak{so}_{2n+1}$ and $\mathfrak{sp}_{2n}$ we need to define certain fillings of shape $\lambda$, which we now explain.

\subsection{Kashiwara--Nakashima Tableaux}
Consider the alphabets for Lie types $B_n$ and $C_n$ below,
\begin{align*}
\aB_{n} &:= \left\{1 \prec \cdots \prec n \prec 0 \prec \bar n \cdots \prec \bar 1 \right\} \\
\aC_{n} &:= \left\{1 \prec \cdots \prec n \prec \bar n \cdots \prec \bar 1 \right\}.
\end{align*}

Given any filling $T$ of shape $\lambda = (\lambda_1, \lambda_2, \dots,\lambda_k)$ in an alphabet $\mathsf{A}$, the \newword{reading word} of $T$, denoted $w(T)$, is the word in $\mathsf{A}$ obtained by reading the entries of $T$ down each column, from  right  to left.

\begin{definition}\label{def:colTab}
A \newword{type $B_n$ column Kashiwara--Nakashima (KN) tableau of height $k$} is a filling $C$ of the  shape $(1^k)$ for some $1\leq k \leq n$ with entries in $\aB_n$ such that
\begin{itemize}
\item[$(i)$] all entries in $C$ are strictly increasing from top to bottom with the sole exception that $0$ may be repeated, and
\item[$(ii)$] if both $z$ and $\bar z$ appear in $C$ with $z$ in the $p^{th}$ box from the top and $\bar z$ in the $q^{th}$ box from the bottom, then $N(z):= p+q \leq z$.
\end{itemize}
Denote by $\KNb{n}{(\emptyset | 1^k)}$ the set of all type $B_n$ column KN tableaux of shape $(1^k)$.
\end{definition}

We define the \textbf{weight} of a type $B_n$ column Kashiwara--Nakashima tableau $C$ to be the integer vector
\[\wt(C) = (b_1,\dots, b_n),\]
where $b_i$ equals the number of $i$'s minus the number of $\bar{i}$'s that appear in $C$, for each { $1\le i\le n$}.

Whenever both $z$ and $\bar z$ appear in $C$ we say that $C$ \emph{contains the pair} $(z,\bar z)$. When $z=0$, we consider each individual $0$ a pair with itself.

Given $C \in \KNb{n}{(\emptyset | 1^k)}$, consider the subset
\begin{equation}\label{eq:I(C)}
I(C):= \{ z_1,\dots,z_s\} \subseteq \{ 1,\dots, n\}
\end{equation}
consisting of all letters $z_i\prec 0$ such that $(z_i, \bar z_i)$ is a pair in $C$ with $z_{i+1}\prec z_{i}$ for all $1 \leq i \leq s$, or $z_i = 0$. Using this, we construct the set 
\begin{equation}\label{eq:J(C)}
J(C):=(t_1,\cdots, t_s),
\end{equation}
where for each $z_i \in I(C)$ we declare $t_1=\max\{t\in \aB_n: t \prec z_1, t\not\in C\}$
and recursively for $i=2,\cdots,n$ define
\[t_i=\max\{t\in\aB_n: t\not\in C, \bar{t}\not\in C, t\prec t_{i-1}, t\prec z_i\}.\]
In particular, Definition \ref{def:colTab}  guarantees the existence of the set  $J(C)$. 

\begin{definition}
Given $C \in \KNb{n}{(\emptyset | 1^k)}$ with $J(C)$ as in \eqref{eq:J(C)}, let $rC$ and $lC$ be the columns obtained from $C$ by replacing $\bar z_i$ with $\bar t_i$ and $z_i$ with $t_i$, respectively. The \newword{splitting} of $C$ is the tableau of shape $(2^k)$,
\[
\Split(C):=lCrC,
\]
with entries $lC$ in the left column and $rC$ in the right column.
\end{definition}

\begin{example} \label{ex:splitb1}
    Let $k=9\leq n$ and $C \in \KNb{n}{(\emptyset |1^9)}$ be a column with the reading word $w(C)=246900\bar{9}\bar{4}\bar{2}$ and weight $\wt(C)= (0,0,0,0,0,1,0,0,0)$. Then $I(C)=(0,0,9,4,2)$, $J(C)=(8,7,5,3,1)$, $w(lC)=135678\bar{9}\bar{4}\bar{2}$, and $w(rC)=2469\bar{8}\bar{7}\bar{5}\bar{3}\bar{1}$. 
\end{example}

\begin{definition}\label{def:spinTab}
A \newword{spin Kashiwara--Nakashima (KN) tableau of height $n$} is a filling $S$ of  shape $(1^n)$ with entries in $\aC_n$ such that
\begin{itemize}
\item[$(i)$] all entries in $S$ are strictly increasing from top to bottom, and
\item[$(ii)$] $S$ contains no pairs $(z,\bar z)$ for any $z \in \aC_n$.
\end{itemize}
Denote by $\KNb{n}{(1^n|\emptyset)}$ the set of all such spin tableaux.
\end{definition}

We define the \textbf{weight} of a spin Kashiwara--Nakashima tableau $S$ to be the $\frac{1}{2}\mathbb{Z}$-vector
\[\wt(S) = \frac{1}{2}(b_1,\dots, b_n),\]
with the $b_i$'s defined as before. In particular, by condition $(ii)$, the value of $ b_i$ is always $\pm 1$. Similarly, the \newword{splitting} of a spin column $S \in \KNb{n}{(1^n|\emptyset)}$ is the KN tableau $\Split(S) \in \KNb{n}{(\emptyset|1^n)}$ with identical entries as $S$ but whose weight equals $\wt(\Split(S))=2\wt(S)$. In particular, $\Split(\KNb{n}{(1^n|\emptyset)}) \subsetneq \KNb{n}{(\emptyset|1^n)}$.

\begin{remark} Conventionally, spin tableaux (which correspond to the fundamental weight $\omega_n^{B_n}=\frac{1}{2}(1^n)$) are denoted by half-width columns to represent their half integer weights \cite{kashiwaranakashima1994crystal,hong2002introduction, lecouvey2003schensted}. In this article we instead denote the spin tableaux  using \emph{gray columns} of regular width (as opposed to white columns) to increase legibility of the labels.
\end{remark}

\begin{example}Let $n=2$.  The set of all spin Kashiwara--Nakashima tableaux $\KNb{2}{(1^2|\emptyset)}$ and their respective weights are:
    $$
    \spintab{1\\2}
 \quad\left(\frac{1}{2},\frac{1}{2}\right);\quad
   \spintab{1\\\bar2}
  \quad\left(\frac{1}{2},-\frac{1}{2}\right);
  \quad
   \spintab{2\\\bar1}
   \quad\left(-\frac{1}{2},\frac{1}{2}\right);
   \quad
   \spintab{\bar 2\\ \bar1}
   \quad \left(-\frac{1}{2},-\frac{1}{2}\right).
   $$
   Their splittings and respective weights are thus given by the following KN  tableaux of shape $(1^2)$: 
$$
    \tableau{1\\2}
 \quad\left(1,1\right);\quad
   \tableau{1\\\bar2}
  \quad\left(1,-1\right);
  \quad
   \tableau{2\\\bar1}
   \quad\left(-1,1\right);
   \quad
   \tableau{\bar 2\\ \bar1}
   \quad \left(-1,-1\right).
   $$
\end{example}

More generally, we will be interested in tableaux of shapes $\lambda$ comprised of a gray spin column $\mu_0=(1^n)$ alongside an $n$-part partition $\mu=(\mu_1,\dots,\mu_n)$. Thus, we will refer to any $\lambda \in P^+$ as a \newword{spin partition} and write $\lambda = (\mu_0 | \mu)$, where either $\mu_0$ or $\mu$ are allowed to be the empty partition. 

\begin{example} Let $n=3$. Then the spin partition $\lambda=(1^3|4,2,1)$ is represented by the following diagram that concatenates the spin column $(1^3)$ with the partition $(4,2,1)$:
$$ \spintab{{}\\{}\\ {} }\tableau{{}&{}&{}&{}\\ {}& {}\\ {}}.$$
\end{example}

Let $\tilde \mu=(\tilde \mu_1, \dots, \tilde \mu_{\mu_1})$ denote the conjugate partition of $\mu$ so that the $i^{th}$ column of the Young diagram of $\mu$ has height $\tilde \mu_i$. Combining Definitions \ref{def:colTab} and \ref{def:spinTab} yields the following.

\begin{definition}\label{def:KNtableau} Given a spin partition $\lambda= (\mu_0| \mu)$, 
a type $B_n$ \newword{Kashiwara--Nakashima tableau of shape $\lambda$} is a filling $T$ with entries in $\aB_n$ of the diagram of $\lambda$ such that, if $T_{i}$ denotes the $i^{th}$ column of $T$, then:
\begin{itemize}
\item $T_0 \in \KNb{n}{(1^n| \emptyset)},$ 
\item $T_i \in \KNb{n}{(\emptyset |1^{\tilde \mu_i})},$ 
\item every row in $T$ is weakly increasing from left to right and has no repeated zeros, and
\item $\Split(T) := \Split(T_0)\; \Split(T_1)\cdots \Split(T_{\mu_1})$ is a semistandard tableau.
\end{itemize}
 We denote by $\KNb{n}{(\lambda)}= \KNb{n}{(\mu_0| \mu)}$ the set of all type $B_n$ Kashiwara--Nakashima tableaux of shape $\lambda$. 
\smallskip

Given a partition $\mu$, we also define the \newword{type $C_n$ Kashiwara--Nakashima (KN) tableaux of shape $\mu$} as the subset $\KNc{n}{(\mu)}$ of $\KNb{n}{(\emptyset | \mu)}$ consisting of all tableaux with entries exclusively in $\aC_n$. 
\end{definition}

 We note that any $T\in \KNb{n}{(\mu_0 | \mu)}$ satisfies the condition that $\Split(T)\in \KNc{n}{(\mu_0+2\mu)}$.
\smallskip

The \newword{weight} of $T \in \KNb{n}{(\lambda)}$ is the vector $\wt(T) = \sum_{i = 0}^{k} \wt(T_i).$ Consequently, $\wt(\Split(T))=2\wt(T).$
\medskip

The previous definitions extend to skew shapes \cite{lecouvey2003schensted}. A \newword{skew orthogonal tableau} ${T}$  is a skew Young diagram (with potentially a leftmost shaded spin column) filled by letters of $\aB_n$ whose columns are admissible of type $B$ and such that the rows of $\Split(T)$ are weakly increasing from left to right.

\begin{example}\label{ex:splitb2}
Let $n=3$ and consider $T \in \KNb{3}(1^3| 3,2,1)$ as below with weight 
 $\wt(T)=\frac{1}{2}(1,-1,1)+ (0,0,1)+(0,0,0)+(0,0,-1)=\frac{1}{2}(1,-1,1)$. Then $\Split(T) \in \KNc{3}{(7,5,3)}$ is given by,
\begin{align*}
T=\spintab{1\\3\\ \bar 2}\tableau{2&0&\bar 3\\ 3& 0\\ \bar 2}\xrightarrow{\Split}
{\Split(T)}=\Split(T_0)\cdots\Split(T_3) =\tableau{ 1&1&2&2&\bar 3&\bar 3&\bar 3\\ 3&3& 3&3&\bar 2\\\bar 2& \bar 2&\bar 1}.
\end{align*}
In particular, $\wt(\Split(T)) = 2(1,-1,1)+(-1,1,1)+(0,1,1)+(0,-1,-1)+2(0,0,-1)=(1,-1,1).$
\end{example}

\subsection{Kashiwara--Nakashima Crystal Operators}
The crystal graph $\B(\lambda)$ of a highest weight $\mathfrak{so}_n$ or $\mathfrak{sp}_{2n}$-representation $V(\lambda)$ has a realization on KN tableaux. We recall this crystal structure on $\KNb{n}{(\lambda)}$ by first defining it on the crystal of reading words of $\KNb{n}{(\lambda)}$.

Given $T \in \KNb{n}{(\mu_0| \mu)}$ let $w(T) = x_1\cdots x_m|z_1\cdots z_n$ be the reading word of $T$ where $z_1\dots z_n$ is the reading word of the spin column $T_0 \in \KNb{n}{(\mu_0| \emptyset)}$ and $x_1\dots x_m$ the reading word of the non-spin columns $T_1\dots T_{\mu_1}$ of $T$. 
For each $i \in I$, the \newword{signature} $\sigma_i$ of the entries of $w(T)= x_1\cdots x_m|z_1\cdots z_n$ is determined by the following rules:

\begin{itemize}[leftmargin=10mm]
 \item[(B)]
 For non-spin entries $x_j$: $\label{nonspinsignaturebn}\begin{cases} \sigma_{i}(i) = + \hbox{ and } \sigma_{i}(\bar i) = - & \; \hbox{ for } i \neq n, x_j=i,\bar{i}; \\
\sigma_{i}(i+1) = - \hbox{ and } \sigma_{i}(\overline{i+1}) = + & \; \hbox{ for } i \neq n, x_j=i+1,\overline{i+1};\\
\sigma_{n}(n) = ++ , \sigma_{n}(0) = -+ \;\hbox{ and }
\sigma_{n}(\bar n) =  -- & \; \hbox{ for } i=n, x_j=n,0,\bar{n}. 
\end{cases}
$
\smallskip
\item[(C)]
 For spin entries $z_j$:
$\label{Csignature}
\begin{cases} \sigma_{i}(i) = + \hbox{ and } \sigma_{i}(\bar i) = - & \hspace{2mm} \hbox{ for all } i, z_j=i,\bar{i}; \\
\sigma_{i}(i+1) = - \hbox{ and } \sigma_{i}(\overline {i+1}) = + & \hspace{2mm} \hbox{ for } i \neq n,z_j=i+1,\overline{i+1} .\\
\end{cases}
$
\end{itemize}
In particular, $\sigma_i$ assigns nothing to an entry $y$ whenever $y\neq i,i+1,\bar i, \overline{i+1}$ with $i\neq n$ 
or when $y\neq 0,n,\bar n$ with $i=n$. Thus, to any word $w(T)$ we can assign a sequence $\sigma_i(w(T))$ with values in $\{+,-,\cdot\}$, where $\cdot$ is used to keep track of the positions where $\sigma_i$ returns an empty value.

We note that the signature rule for spin entries above is nothing more than the type $C_n$ signature rule. Thus, for a type $C_n$ KN tableau $T$ we can define the signature of $w(T)$ similarly but use (C) for every entry instead.

\begin{definition}
For $T \in \KNb{n}{(\lambda)}$ and $i \in I$, define the \newword{i-pairing} of entries $w(T)$ as follows. We first assign an $i$-pairing on the entries of $\sigma_i(w(T))$ by  iteratively $i$-pairing any unpaired $+$ with the leftmost $-$ to its right provided all other signs $\{+,-\}$ between them have already been $i$-paired, and ignoring any $\cdot$'s in the process. The $i$-pairing on $w(T)$ is the pairing induced on its entries from the $i$-pairing of their corresponding signatures. 
\smallskip

We say an entry $y$ in $w(T)$ is \newword{unpaired} if $\sigma_i(y)=+$ and after $i$-pairing the entries of $w(T)$ as above, it is not $i$-paired with an entry to its right. 
\end{definition}

\begin{definition} \label{def:snale}
For each $i \in I$ and $T \in \KNb{n}{(\lambda)}$, define the \newword{lowering operator} $ f^B_i$ on each $w(T)= x_1\dots x_m|z_1\dots z_n$ as follows:
\begin{itemize}
\item If $w(T)$ has no unpaired entries (with respect to $\sigma_i$), then $ f^B_i(w(T))=0$.
\item Otherwise, $f^B_i$ will act on the \emph{leftmost} unpaired entry $y$ of $w(T)$ via the following assignment:
\begin{itemize}
\item If $y$ lies in a nonspin column then $ f^B_i$ will send $y=i \mapsto i+1$ or $y=\overline{i+1} \mapsto \bar i$ if $ i \neq n$ or send $y=n \mapsto 0$ or $y=0 \mapsto \bar n$ if $i=n$, and leave all other entries unchanged.
\item If $y$ lies in a spin column and $y=i\neq n$ with an entry $z=\overline{i+1}$ below it, then $f^B_i$ will simultaneously send $y=i\mapsto i+1$ and $z=\overline{i+1}\mapsto \bar i$. If instead $y=i\neq n$ but $\overline{i+1}$ is not contained in the spin column, then $f^B_i(w(T)) =0$. Lastly, if $i=n$ then $f^B_n$ maps $y=n \mapsto \bar n$, leaving all other entries unchanged.
\end{itemize}
\end{itemize}
Similarly, for $T \in \KNc{n}{(\mu)}$ the lowering operator $ f^C_i$ acts on the leftmost unpaired entry of $w(T)$ as in the non-spin case above with the only change being that for $i=n$, $ f^C_n$ maps $n\mapsto \bar n$. In the non-spin case, $f_n^Bf_n^B(n)=f_n^C(n)$.
\end{definition}

Thus, define the lowering operators $f_i$ on $T \in \KNb{n}(\lambda)$ as the induced operators from those on $w(T)$, where if the entry $y$ in $w(T)$ is modified under the action of $f_i$, then the corresponding entry in $T$ is modified in the same way. It is a theorem of Kashiwara and Nakashima \cite{kashiwaranakashima1994crystal} that these operators endow the sets $\KNb{n}(\lambda)$ (resp. $\KNc{n}(\lambda)$) with the structure of an $\mathfrak{so}_{2n+1}$-crystal (resp. $\mathfrak{sp}_{2n}$-crystal). Thus, we write $\B(\mu_0|\mu)$ for the $\mathfrak{so}_{2n+1}$-crystal on $\KNb{n}(\mu_0|\mu)$ with highest weight $\lambda=(\mu_0|\mu)$, and  $\C(\mu)$ for the $\mathfrak{sp}_{2n}$-crystal on $\KNc{n}(\mu)$ with highest weight $\lambda=\mu$.

\begin{example} \label{BCwords}
Suppose $n=8$ with
$T = \hackcenter{\tableau{8\\0\\0}}$ 
so that $w(T) = 800$. 
Below we see the action of $f_8^B$ on
$w(T)$ and $T$. 
At each step the corresponding signatures $\sigma_8$ are provided, with the $8$-paired entries connected in red, and the unpaired entry circled on which $f_8^B$ acts in blue. 
\[
\hackcenter{\begin{tikzpicture}[xscale=.7]
\node at (-2,0) {$w(T)$};
\node at (-2,.5){$\sigma_8(w)$};
\node at (-2,-.5) {$T$};
\node at (0,.5) {$++$};
\node at (0,0) {8};
\node at (1,.5) {$-+$};
\node at (1,0) {0};
\node at (1,-1){$\tableau{8\\0\\0}$};
\node at (2,.5) {$-+$};
 \node at (2,0) {0}; 
 \draw[red](.2,.65)--(.2,.75)--(.8,.75)--(.8,.65);
  \draw[red](1.2,.65)--(1.2,.75)--(1.8,.75)--(1.8,.65);
  \draw[blue] (-.2,.5) circle (6pt); 
    \draw[blue] (0,0) circle (6pt); 
     \draw[blue] (1,-.54) circle (6pt); 
\end{tikzpicture}}
\xrightarrow{f^B_8}
\hackcenter{
\begin{tikzpicture}[xscale=.7]
\node at (0,.5) {$-+$};
\node at (0,0) {0};
\node at (1,.5) {$-+$};
\node at (1,0) {0};
\node at (1,-1){$\tableau{0\\0\\0}$};
\node at (2,.5) {$-+$};
 \node at (2,0) {0}; 
 \draw[red](.2,.65)--(.2,.75)--(.8,.75)--(.8,.65);
  \draw[red](1.2,.65)--(1.2,.75)--(1.8,.75)--(1.8,.65);
  \draw[blue] (2.2,.5) circle (6pt); 
   \draw[blue] (2,0) circle (6pt); 
   \draw[blue] (1,-1.45) circle (6pt); 
\end{tikzpicture}}
\xrightarrow{f^B_8}
\hackcenter{
\begin{tikzpicture}[xscale=.7]
\node at (0,.5) {$-+$};
\node at (0,0) {0};
\node at (1,.5) {$-+$};
\node at (1,0) {0};
\node at (1,-1){$\tableau{0\\0\\\bar8}$};
\node at (2,.5) {$--$};
 \node at (2,0) {$\bar8$}; 
 \draw[red](.2,.65)--(.2,.75)--(.8,.75)--(.8,.65);
  \draw[red](1.2,.65)--(1.2,.75)--(1.8,.75)--(1.8,.65); 
\end{tikzpicture}}
\xrightarrow{f^B_8}
0.
\]
\end{example}

\begin{example} Continuing with Example \ref{ex:splitb2}, with $n=3$ and $T \in \KNb{3}(1^3| 3,2,1)$. As above, we illustrate the action of $f_3^B$ on $\sigma_3(w)$, $w(T)$, and $T$.
\[
\hackcenter{\begin{tikzpicture}[xscale=1]
\node at (-.5,.5){$\sigma_3(w)$};
\node at (-.5,0) {$w(T)$};
\node at (-.5,-.5) {$T$};
\node at (1,.5) [scale=.8]{$--$};
\node at (1,0) {$\bar 3$};
\node at (1.5,.5) [scale=.8]{$-+$};
\node at (1.5,0) {0};
\node at (2,.5) [scale=.8]{$-+$};
 \node at (2,0) {0}; 
\node at (2.5,.5) {$\cdot$};
 \node at (2.5,0) {2};
 \node at (3,.5) [scale=.8]{$++$};
 \node at (3,0) {3}; 
 \node at (3.5,.5) {$\cdot$};
 \node at (3.5,0) {$\bar2$}; 
 \node at (4,.5) {$|$};
 \node at (4,0) {$|$};
 \node at (4.5,.5) {$\cdot$};
 \node at (4.5,0) {$1$}; 
 \node at (5,.5) {$+$};
 \node at (5,0) {3}; 
 \node at (5.5,.5) {$\cdot$};
 \node at (5.5,0) {$\bar 2$}; 
\node at (3,-1){$\spintab{1\\3\\ \bar 2}\tableau{2&0&\bar 3\\ 3& 0\\ \bar 2}$};
  \draw[red](1.6,.65)--(1.6,.75)--(1.9,.75)--(1.9,.65);
  \draw[blue] (2.1,.5) circle (4pt); 
    \draw[blue] (2,0) circle (6pt); 
     \draw[blue] (3.22,-1) circle (6pt); 
\end{tikzpicture}}
\xrightarrow{f_3^B}
\hackcenter{\begin{tikzpicture}[xscale=1]
\node at (1,.5) [scale=.8]{$--$};
\node at (1,0) {$\bar 3$};
\node at (1.5,.5) [scale=.8]{$-+$};
\node at (1.5,0) {0};
\node at (2,.5) [scale=.8]{$--$};
 \node at (2,0) {$\bar 3$}; 
\node at (2.5,.5) {$\cdot$};
 \node at (2.5,0) {2};
 \node at (3,.5) [scale=.8]{$++$};
 \node at (3,0) {3}; 
 \node at (3.5,.5) {$\cdot$};
 \node at (3.5,0) {$\bar2$}; 
 \node at (4,.5) {$|$};
 \node at (4,0) {$|$};
 \node at (4.5,.5) {$\cdot$};
 \node at (4.5,0) {$1$}; 
 \node at (5,.5) {$+$};
 \node at (5,0) {3}; 
 \node at (5.5,.5) {$\cdot$};
 \node at (5.5,0) {$\bar 2$}; 
\node at (3,-1){$\spintab{1\\3\\ \bar 2}\tableau{2&0&\bar 3\\ 3& \bar 3\\ \bar 2}$};
  \draw[red](1.6,.65)--(1.6,.75)--(1.9,.75)--(1.9,.65);
  \draw[blue] (2.9,.5) circle (4pt); 
    \draw[blue] (3,0) circle (6pt); 
     \draw[blue] (2.75,-1) circle (6pt); 
\end{tikzpicture}}
\]
\[
\xrightarrow{f_3^B}
\hackcenter{\begin{tikzpicture}[xscale=1]
\node at (1,.5) [scale=.8]{$--$};
\node at (1,0) {$\bar 3$};
\node at (1.5,.5) [scale=.8]{$-+$};
\node at (1.5,0) {0};
\node at (2,.5) [scale=.8]{$--$};
 \node at (2,0) {$\bar 3$}; 
\node at (2.5,.5) {$\cdot$};
 \node at (2.5,0) {2};
 \node at (3,.5) [scale=.8]{$-+$};
 \node at (3,0) {0}; 
 \node at (3.5,.5) {$\cdot$};
 \node at (3.5,0) {$\bar2$}; 
 \node at (4,.5) {$|$};
 \node at (4,0) {$|$};
 \node at (4.5,.5) {$\cdot$};
 \node at (4.5,0) {$1$}; 
 \node at (5,.5) {$+$};
 \node at (5,0) {3}; 
 \node at (5.5,.5) {$\cdot$};
 \node at (5.5,0) {$\bar 2$}; 
\node at (3,-1){$\spintab{1\\3\\ \bar 2}\tableau{2&0&\bar 3\\ 0& \bar 3\\ \bar 2}$};
  \draw[red](1.6,.65)--(1.6,.75)--(1.9,.75)--(1.9,.65);
  \draw[blue] (3.1,.5) circle (4pt); 
    \draw[blue] (3,0) circle (6pt); 
     \draw[blue] (2.75,-1) circle (6pt); 
\end{tikzpicture}}
\xrightarrow{f_3^B}
\hackcenter{\begin{tikzpicture}[xscale=1]
\node at (1,.5) [scale=.8]{$--$};
\node at (1,0) {$\bar 3$};
\node at (1.5,.5) [scale=.8]{$-+$};
\node at (1.5,0) {0};
\node at (2,.5) [scale=.8]{$--$};
 \node at (2,0) {$\bar 3$}; 
\node at (2.5,.5) {$\cdot$};
 \node at (2.5,0) {2};
 \node at (3,.5) [scale=.8]{$--$};
 \node at (3,0) {$\bar 3$}; 
 \node at (3.5,.5) {$\cdot$};
 \node at (3.5,0) {$\bar2$}; 
 \node at (4,.5) {$|$};
 \node at (4,0) {$|$};
 \node at (4.5,.5) {$\cdot$};
 \node at (4.5,0) {$1$}; 
 \node at (5,.5) {$+$};
 \node at (5,0) {3}; 
 \node at (5.5,.5) {$\cdot$};
 \node at (5.5,0) {$\bar 2$}; 
\node at (3,-1){$\spintab{1\\ 3\\ \bar 2}\tableau{2&0&\bar 3\\ \bar 3& \bar 3 \\ \bar 2}$};
  \draw[red](1.6,.65)--(1.6,.75)--(1.9,.75)--(1.9,.65);
  \draw[blue] (5,.5) circle (4pt); 
    \draw[blue] (5,0) circle (6pt); 
     \draw[blue] (2.3,-1) circle (6pt); 
\end{tikzpicture}}
\]
\[
\xrightarrow{f_3^B}
\hackcenter{\begin{tikzpicture}[xscale=1]
\node at (1,.5) [scale=.8]{$--$};
\node at (1,0) {$\bar 3$};
\node at (1.5,.5) [scale=.8]{$-+$};
\node at (1.5,0) {0};
\node at (2,.5) [scale=.8]{$--$};
 \node at (2,0) {$\bar 3$}; 
\node at (2.5,.5) {$\cdot$};
 \node at (2.5,0) {2};
 \node at (3,.5) [scale=.8]{$--$};
 \node at (3,0) {$\bar 3$}; 
 \node at (3.5,.5) {$\cdot$};
 \node at (3.5,0) {$\bar2$}; 
 \node at (4,.5) {$|$};
 \node at (4,0) {$|$};
 \node at (4.5,.5) {$\cdot$};
 \node at (4.5,0) {$1$}; 
 \node at (5,.5) {$-$};
 \node at (5,0) {$\bar 3$}; 
 \node at (5.5,.5) {$\cdot$};
 \node at (5.5,0) {$\bar 2$}; 
\node at (3,-1){$\spintab{1\\ \bar 3\\ \bar 2}\tableau{2&0&\bar 3\\ \bar 3& \bar 3 \\ \bar 2}$};
  \draw[red](1.6,.65)--(1.6,.75)--(1.9,.75)--(1.9,.65);
\end{tikzpicture}}
\]
\end{example}

\subsection{Symplectic Jeu de Taquin}\label{sjdt}
In this section we briefly recall the symplectic jeu de taquin procedure \cite{sheats1999symplectic,lecouvey2002schensted, sa21b}.

\begin{definition}
    Given a column $C \in \KNc{n}{(1^k)}$ with $\Split(C) = lCrC$, let $\Phi(C)$ be the tableau of shape $(1^k)$ obtained by combining the unbarred entries from $lC$ and barred entries from $rC$.
\end{definition} 
{ In particular, $\Phi(C) = C$ if and only if  $C$ does not have symmetric entries.   The column $\Phi(C)$ is a \emph{co-admissible column} and the procedure to form $\Phi(C)$ from $C$ is reversible (for details see \cite[Section 2.2]{lecouvey2002schensted, sa21a} or \cite[Section 2]{azenhas2024virtualkeys}). In particular, every column on the alphabet $[n]$ is simultaneously admissible and co-admissible. The map $\Phi$ is a bijection between admissible and co-admissible columns of the same height.}
\begin{example}
    Consider the admissible column $C$ of shape $(1^3)$ with reading word $w(C)=23\bar2$, so that $w(lC)= 13\bar 2$, $w(rC)=23\bar1$. Then, $\Phi(C)$ is the column of shape $(1^3)$ with reading word $w(\Phi(C))=13\bar1$. { Also consider the column $C'$ with reading word $w(C')=24\bar2$ which is simultaneously admissible and co-admissible. Then $\Phi(C')$ is the column of shape $(1^3)$ with reading word $w(\Phi(C'))=14\bar1$ and $w(\Phi^{-1}(C'))=34\bar3$. }
\end{example}

 Let $T$ be a punctured symplectic KN skew tableau  \footnote{In the literature  the punctured cell is often assigned with the symbol $*$. Here is just a blank cell as in JDT, and the symbol $*$ in a box indicates a (potentially) filled box.} with two columns $C_1$ and $C_2$ with the puncture
in $C_1$  and $\Split(T)=l C_1 rC_1l C_2rC_2$.
Let $\alpha$ be the entry under the puncture of $rC_1$, and $\beta$ the entry to the right of the puncture of $rC_1$, that is, locally, the tableaux $T=C_1C_2$ and $\Split(T) = l C_1 rC_1l C_2rC_2$ look like:
\[
\Skew(1: b | 0: a, \ast)
\qquad
\xrightarrow{\Split}
\qquad
\Skew(2: \beta, \ast | 0: \ast,\alpha, \ast, \ast)\;,
\]
where either $a$ or $b$ may not exist. The elementary steps of the symplectic jeu de taquin, or SJDT  for short, are the following:

\begin{itemize} \item[A.] Suppose $a$ exists. If $\alpha\leq \beta$ or $\beta$ does not exist,  then the puncture in $T$ will change its position with the cell beneath it. This is called a \emph{vertical slide}.
\[
\tableau{& b\\ a&\ast}
\qquad
\mapsto
\qquad
\tableau{a& b\\ &\ast}
\]

\item[B.] Suppose $b$ exists. If $\alpha> \beta$ or $\alpha$ does not exist we say the slide is \emph{horizontal}.
Let $C_1'$ and $C_2'$ be the columns of $T$ obtained after the slide. 
\[
\tableau{& b\\ a&\ast}
\qquad
\mapsto
\qquad
\tableau{b' & \\ a' &\ast}
\]
There are two cases depending on the sign of $\beta$:

\begin{enumerate}
    \item
 If $\beta$ is barred then $b=\beta$ and $a'$ and $b'$ are obtained by setting $C_1':=\Phi^{-1}\left(\Phi(C_1)\sqcup \Skew(0:\beta)\right)$ and $C_2':=C_2\setminus\Skew(0:\beta)$.
\item If $\beta$ is unbarred, then set  $C_1':=C_1\sqcup \Skew(0:\beta)$ with $a'=a$, $b'=\beta$, and $C_2':=\Phi^{-1}(\Phi(C_2)\setminus \Skew(0:\beta))$. This case, however, may yield a column $C_1'$ that is no longer a type $C$ KN tableau. When this occurs we modify $C_1'$ as follows:
\subitem{-} Let $i$ be the smallest value such that $C_1'$ contains the pair $(i,\bar i)$ with $N(i)>i$.
\subitem{-} Construct $C_1''$ from $C_1'$ by removing the topmost and bottommost cells from $C_1'$ and labeling the boxes linearly with the entries from $C_1'$, skipping over $i$ and $\bar i$.
\subitem{-} Replace $C_1'$ with $C_1''$. 
\end{enumerate}
\end{itemize}
Iterative application of this procedure will eventually yield a left justified shape, that is for every puncture both $a$ and $b$ will not exists. When this occurs, \textit{symplectic jeu de taquin} ends. 

\begin{remark}
It is important to note that SJDT is reversible by pushing the filled cells to the right and the empty cells to the left. The reverse process is denoted RSJDT.
\end{remark}

The \newword{rectification} $\operatorname{rect}_C(T)$ of a a symplectic KN skew tableau $T$ is the tableau of partition shape obtained by iterated applications of SJDT. In particular, rectification is independent of the order in which the inner corners are filled \cite[Corollary 6.3.9]{lecouvey2002schensted}.

\begin{remark}\label{re:insertionjeutaquin} If the columns $C_1$ and $C_2$ do not have barred entries then the SJDT coincides with the \textit{jeu de taquin} on semistandard Young tableaux.
\end{remark}

\begin{remark} \label{re:tensorinsertionscheme}The rectification map on words, defined  by the $B_n$-spin insertion or $B_n$($C_n$)-insertion scheme, $w\otimes w'\mapsto [\emptyset\leftarrow w\leftarrow w']$ is a  $B_n$ ($C_n$)-crystal isomorphism \cite{baker00bn,lecouvey2003schensted}. In type $C_n$ it coincides with the rectification defined by SJDT on the diagonal skew  tableau with  reading word $w\otimes w'$.
\end{remark}

\begin{example}
The following computation shows an example of rectification using SJDT.
\[
\hackcenter{\tableau{&&3\\&3&\bar 3\\1&\bar 3}}\;\;
\mapsto\;
\hackcenter{\tableau{&2\\&3&\bar 2\\1&\bar 3}}\;\;
\mapsto\;
\hackcenter{\tableau{&2&\bar 2\\&3&\\1&\bar 3}}\;\;
\mapsto\;
\hackcenter{\tableau{&2&\bar 2\\1&3&\\&\bar 3}}\;\;
\mapsto\;
\hackcenter{\tableau{&2&\bar 2\\1&3&\\\bar 3}}\;\;
\mapsto\;
\hackcenter{\tableau{1&2&\bar 2\\&3&\\\bar 3}}\;\;
\mapsto\;
\hackcenter{\tableau{1&2&\bar 2\\3\\\bar 3}}\;\;
\] 
We spell out a few steps in detail. To compute the first step, we begin by splitting the second and third column with $a=b=3$:
\[
\hackcenter{
\tableau{&3\\3&\bar3\\\bar 3}
}
\xrightarrow{\Split}
\hackcenter{
\tableau{&&2&3\\2&3&\bar3&\bar2\\\bar3&\bar 2}}.
\]
Hence, $\alpha=3$ and $\beta=2$ with $\beta$ unbarred, so this is case B(2). Thus, 
\[C_1'= \tableau{\\3\\\bar 3}\sqcup \tableau{2} = \tableau{2\\3\\ \bar3}
\qquad \text{and} \qquad
C_2' = \Phi^{-1}\left(\tableau{2\\\bar2} \setminus \tableau{2}\right) = \tableau{\\ \bar2}.\]
Hence, replacing the second column of $T$ by $C_1'$ and the third column with $C_2'$ yields $\tableau{&2\\&3&\bar 2\\1&\bar 3}$.

Similarly, in step four we first split the first and second columns of the tableau:
\[
\hackcenter{
\tableau{&2\\1&\bar3\\&\bar3}
}\;\;
\xrightarrow{\Split}\;
\hackcenter{
\tableau{&&1&2\\1&1&2&3\\&&\bar3&\bar1}
}
\]
so that $\alpha$ is nonexistent and $\beta = \bar 3$. Since $\beta$ is barred, this falls under case B(1). Thus, 
\[C_1'=\Phi^{-1}\left(\tableau{1\\} \sqcup \tableau{ \\ \bar3}\right) = \tableau{1\\ \bar 3}
\qquad \text{and} \qquad
C_2' = \tableau{2\\ \bar 3 \\ \bar 3}\setminus \tableau{\\ \\ \bar 3} = \tableau{2\\\bar3 \\}
\]
which yields the skew-tableau $\hackcenter{\tableau{&2&\bar 2\\1&3&\\\bar 3}}$.
\end{example}

\subsection{Virtualization and Jeu de Taquin in Type B}
We build on the virtualization procedure described in \cite{pappe2023promotion}  where virtual crystals are provided for spin and vector (or standard) representation
of type $B_n$ into type $C_n$.  They are crystal isomorphic to those obtained in the tableau crystal by the
splitting
map, defined above,  viewed as a virtualization map from tableau crystals of type $B_n$ to ones of type $C_n$. Recall that for a chosen Dynkin embedding the virtualization map is unique    and by
 Proposition \ref{virtualproperties} virtualization is closed for the tensor product. 

In this section we will use the Dynkin diagram embeddings $\psi_{BC}:B_n \to C_n$ and $\psi_{CB}:C_n \to B_n$
considered in \cite{fujita18, pappe2023promotion} (see Table \ref{fig:gamma}). It follows from their definition that the action of both these virtualization coincides on the spin and standard representation. Since virtualization commutes with the tensor product (see Proposition \ref{virtualproperties}), then the virtualizations in \cite{fujita18, pappe2023promotion} are the same. Thus, we will use both definitions intercheangably and denote them $\vt_{BC}$ and $\vt_{CB}$.

We will also make ample usage of the \emph{symplectic insertion} \cite{baker00bn, lecouvey2002schensted, KimInsertion} and \emph{orthogonal insertion} \cite{lecouvey2003schensted, KimInsertion} on skew tableaux, referring the readers to these sources for details. 

The following proposition can also be found in \cite[Sec. 6]{fujita18}. 

\begin{proposition}
The composition of the virtualization maps $\vt_{BC}$ and $\vt_{CB}$ yield $2$-dilations. Namely, 
\[ \vt_{CB} \circ \vt_{BC} = \Theta_2^B \qquad \text{and} \qquad
\vt_{BC} \circ \vt_{CB} = \Theta_2^C.
\]
\end{proposition}
\begin{proof}
    From Theorem \ref{thm:dilation is virt} we know that any $m$-dilation is a virtualization. Since the compositions of virtualizations is again a virtualization and by Remark \ref{rem:virtual hw}, virtualizations map highest weight vectors to highest weight vectors, it suffices to show that the images of the crystal operators under both maps coincide.

    So suppose $\vt_{BC}: \B \to \V \subset \tilde{B}$ and $\vt_{CB}: \tilde{\B} \to \tilde{\V}$. Then for any $b \in \B$ with $f_i(b) \neq 0$ if $i\neq n$, we have that:
    \[\vt_{CB}( \vt_{BC}(f_i(b)) )= \vt_{CB}(f_i^2(\vt_{BC}(b))) = f_i^2 (\vt_{CB}(\vt_{BC}(b))).\] If instead $i=n$, then:
    \[\vt_{CB}( \vt_{BC}(f_n(b)) )= \vt_{CB}(f_n(\vt_{BC}(b))) = f_n^2 (\vt_{CB}(\vt_{BC}(b))).\]
Thus,$\vt_{CB} \circ \vt_{BC}$ coincides with the $2$-dilation map $\Theta_2$ in type $B_2$.  The reverse composition follows similarly.
\end{proof}

It is important to note that although the images of $\Theta_2(\B(\lambda))$ and $\vt_{BC}(\B(\lambda))$ coincide as graphs, they are not equivalent as crystal embeddings. Indeed, $\Theta_2$ yields an embedding from type $B_n$ onto itself, unlike $\vt_{BC}$ which embeds $\B(\lambda)$ onto a type $C_n$ crystal.
 
\begin{example}\label{ex:spinBC} Let $n=1$. Consider the $B_1$-spin crystal   $  \B(1|\emptyset):= \;\;
   \hackcenter{ \spintab{1}
   \xrightarrow{f_1^B}
\spintab{\bar 1}}.$

Then, its 2-dilation yields the embedding
$
\Theta_2(\B(1|\emptyset)):= \boxed{1} \;\;{\color{orange}\xrightarrow{f_1^Bf_1^B }}\;\;\boxed{\overline{1}}
\subseteq
\boxed{1} \;\;{\color{orange}\xrightarrow{f^B_1}}\;\;\boxed{{0}}\;\;{\color{orange}\xrightarrow{f^B_1}}\;\;\boxed{\overline{1}}=: \B(0|1).
$
On the other hand, $\vt_{BC}$ yields the $C_1$-crystal isomorphism $\vt_{BC}(\B(1|\emptyset)) = \C(1):=\boxed{1} \;\;{\xrightarrow{f_1^C }}\;\;\boxed{\overline{1}}$.
\end{example}

\begin{figure}[ht]
  \begin{tikzcd}[cramped,column sep=small, row sep=1.7em]
 &&&\zeroonebartwo \arrow[dr, blue, "1" blue]&&&&\\
 & \fcolorbox{purple}{white}{\oneonebartwo} \arrow[r, blue, "1" blue]& \twoonebartwo \arrow[r, blue, "1" blue] \arrow[ru, red, "2" red]   & \fcolorbox{purple}{white}{\twotwobarone} \arrow[r, red, "2" red]& \zerotwobarone \arrow[r, red, "2" red]& \bartwotwobarone \arrow[r, red, "2" red]& \fcolorbox{purple}{white}{\bartwobartwobarone} \arrow[dr, blue, "1" blue] &\\
 \fcolorbox{purple}{white}{\twoonetwo} \arrow[ur, red, "2" red]\arrow[dr, blue, "1" blue]&&&&&&& \fcolorbox{purple}{white}{\baronebartwobarone}\\
 &  \fcolorbox{purple}{white}{\twoonetwo} \arrow[r, red, "2" red]& \zeroonetwo \arrow[r, red, "2" red]&\bartwoonetwo \arrow[r, red, "2" red] \arrow[dr, blue, "1" blue]& \fcolorbox{purple}{white}{\bartwoonebartwo} \arrow[r, blue, "1" blue] &\baroneonebartwo \arrow[r, blue, "1" blue] & \fcolorbox{purple}{white}{\baronetwobarone} \arrow[ur, red, "2" red]&\\
 &&&&\baroneonetwo \arrow[ur, red, "2" red]&&&
  \end{tikzcd}\caption{The $B_2$-crystal 
    $\B(1^2|1)$ with the extremal weight vectors (keys) boxed in red.}
\label{fig:B2crystal}
\end{figure}

We will now show that the virtualization map $\vt_{BC}$ on KN tableaux is nothing more than the splitting map. 

\begin{theorem}\label{thm:virt-column}
    Suppose $T \in \KNb{n}{(\lambda)}$ with $\lambda = (1^n|\emptyset)$ or $\lambda = (\emptyset| 1^k)$ with $1\leq k< n$. Then, $\vt_{BC}(T) = \Split(T)$. 
\end{theorem}

\begin{proof}
We begin by noting that $\Psi_{BC}(i)=i$ for any vertex $i \in B_n$. Thus, the virtualization maps $\omega_i^B \mapsto 2 \omega_i^C$ for $i\neq n$ and $\omega_n^B \mapsto \omega_n^C$. Consequently, for any $\mu = \sum_{i=1}^n{a_i\omega^{B}_i}$ since $\omega_i^C=\omega_i^B$ for $i\neq n$ and $\omega_n^C=2\omega_n^B$, it follows that $\mu \mapsto \sum_{i=1}^{n-1}{2 a_i\omega^{C}_i} + a_n\omega_n^C=\sum_{i=1}^{n-1}{2 a_i\omega^{B}_i} + a_n2\omega_n^B= 2\mu$.

So consider the spin case when $\lambda = (1^n|\emptyset)=\omega_n^B$. Then, for any $T \in \KNb{n}{(1^n|\emptyset)}$, $\vt_{BC}(T)$ is a tableau in $\KNc{n}{(1^n)}$ with weight $2\wt(T)$. Since by definition, a spin KN tableau has no pairs $(z,\bar z)$ for any value of $z$, it follows that every $T \in  \KNb{n}{(1^n|\emptyset)}$, and thus its image in $\KNc{n}{(1^n)}$, is uniquely determined by its weight. Since $\wt(\Split(T)) = 2 \wt(T)$ it follows that $\vt_{BC}(T) = \Split(T)$. 

Consider now the non-spin case $\lambda=(\emptyset|1^k)$ for some $k\neq n$. Given any $T \in \KNb{n}{(\emptyset|1^k)}$, let $T'$ be the KN tableau of skew shape $(k,k-1,\dots,1) / (k-1,k-2,\dots, 1,0)$
with $w(T')=w(T)=x_1\dots x_m$ (see Example \ref{ex: virtualizationBC}). 

Now, recall that the virtualization $\vt_{BC}$, as defined in \cite{pappe2023promotion}, is defined on the standard representation $\B(\omega_1)$ 
via the map
\[
i \mapsto i i, \bar i \mapsto \bar i \bar i, 0 \mapsto n\bar n  \hbox{ for all } 1 \leq i \leq n
\]
and induced on words via the tensor product on crystals by setting \[\vt_{BC}(w(T))=\vt_{BC}(w(x_1\otimes \dots \otimes x_m)) := w(\vt_{BC}(x_1)\otimes \dots \otimes \vt_{BC}(x_m)),\]
so that $\vt_{BC}(T) = [ \emptyset \xleftarrow{C} \vt_{BC}(w(T))]$.
In particular, since $w(T)=w(T')$, then $\vt_{BC}(T)=\vt_{BC}(T')$. Evidently, since $\Split(T')$ is simply a duplication of the columns of $T'$, then 
\[
w(\Split(T')) = w(\vt_{BC}(x_1)\otimes \dots \otimes \vt_{BC}(x_m)) = \vt_{BC}(w(T'))
\]
Moreover, by \cite[Corollary 6.3.9]{lecouvey2002schensted} we know that for any type C KN tableau rectification via SJDT coincides with type C insertion of its reading word. Thus, it follows that:
\begin{equation}
\rect_C(\Split(T')) = [\emptyset \xleftarrow{C} w(\Split(T'))]  = [ \emptyset \xleftarrow{C} \vt_{BC}(w(T'))] = \vt_{BC}(T).    
\end{equation}
On the other hand, by \cite[Proposition 3.5.3]{lecouvey2003schensted}
also have that 
\[
\rect_C(\Split(T')) = \Split[\emptyset \xleftarrow{B} w(T')] = \Split[\emptyset \xleftarrow{B} w(T)]=\Split(T),
\]
so that $\vt_{BC}(T) = \Split(T)$ as desired. 
\end{proof}

\begin{example}\label{ex: virtualizationBC} 
Let $n=3$ and
    consider the type $B_3$ column tableau $T = \tableau{0\\0\\ \bar{1}}$ with $\Split(T) = \tableau{2&\bar 3\\3&\bar 2\\\bar{1} &\bar{1}}$. Then we have, 
    \[
T'= \tableau{ && 0\\ &0\\ \bar{1}} \qquad \text{and} \qquad
\Split(T')= \tableau{ && &&3&\bar 3\\ &&3&\bar 3\\ \bar{1}&\bar{1}}.
    \]
Computing, we can see that $\vt_{BC}(T) = [\emptyset \xleftarrow{C} \vt_{BC}(00\bar{1})]=[\emptyset \xleftarrow{C} 3\bar 3 3\bar 3\bar{1}\bar{1}]$. From Remark \ref{re:insertionjeutaquin}, we know  that $[\emptyset \xleftarrow{C} 3\bar{3}3\bar{3}\bar 1\bar 1]$ equals  to the SJDT rectification of $\Split(T')$. We then obtain by SJDT:
\begin{align*}
&\tableau{ && &&3&\bar 3\\ &&3&\bar 3\\ \bar{1}&\bar{1}} 
\longrightarrow
~\tableau{ &&&& 3&\bar 3\\ &3&&\bar 3\\ \bar{1}&\bar{1}}
\longrightarrow 
~\tableau{ && &3&\bar 3\\ &3&\bar 3\\ \bar{1}&\bar{1}}
\longrightarrow
~\tableau{ &&& 3&\bar 3\\ 3&&\bar 3\\ \bar{1}&\bar{1}}
\\
\\
&\longrightarrow
~\tableau{ && 3&\bar 3\\ 3&\bar 3&&\\ \bar{1}&\bar{1}}
\longrightarrow 
~\tableau{ & 3&\bar 3\\ 3&\bar 3&\\ \bar{1}&\bar{1}}
\longrightarrow
~\tableau{2&& \bar 3\\ 3&\bar 2&\\ \bar{1}&\bar{1}}
\longrightarrow 
~\tableau{2 & \bar 3\\ 3&\bar 2&\\ \bar{1}&\bar{1}}=\Split(T)
\end{align*}
so that, indeed, $\vt_{BC}(T) = \rect_C(\Split(T'))$.
\end{example}

\begin{corollary}\label{cor:virt-nospin}
    For any $T \in \KNb{n}{(\mu_0| \mu)}$ we have that $\vt_{BC}(T)= \Split(T)$.
    \end{corollary}

\begin{proof} 
Recall that by definition $\Split(T) = \Split(T_0)\Split(T_1)\cdots \Split(T_{\mu_1})$ where $T_i$ is the $i^{th}$ column of $T$, with $T_0$ the (potentially empty) spin column. Since by Theorem \ref{thm:virt-column} 
    \[
\Split(T) = \Split(T_0)\Split(T_1)\cdots \Split(T_{\mu_1}) = \vt_{BC}(T_0)\vt_{BC}(T_1) \cdots \vt_{BC}(T_{\mu_1})
    \]
it follows that, 
     \begin{align*}
w(\Split(T)) &= w(\vt_{BC}(T_0) \vt_{BC}(T_1) \cdots \vt_{BC}(T_{\mu_1})) = w( \vt_{BC}(T_{\mu_1})) \cdots w(\vt_{BC}(T_1))w(\vt_{BC}(T_0))\\
&= \vt_{BC}(w(T_{\mu_1})) \cdots \vt_{BC}(w(T_1)) \vt_{BC}(w(T_0))
= \vt_{BC}(w(T_{\mu_1})\cdots w(T_1)w(T_0)) \\
&= \vt_{BC}(w(T))
    \end{align*}
    Thus, the result follows by performing type C insertion:
    \[ 
\Split(T) = [\emptyset \xleftarrow{C} w(\Split(T))] = [\emptyset \xleftarrow{C} \vt_{BC}(w(T))] = \vt_{BC}(T).
    \]
    \end{proof}

\begin{figure}[ht]
 \begin{tikzcd}[cramped,column sep=small, row sep=1.7em]
 &&&\vzeroonebartwo \arrow[dr, blue, "1" blue]&&&&\\
 & \voneonebartwo \arrow[r, blue, "1" blue]& \vtwoonebartwo \arrow[r, blue, "1" blue] \arrow[ru, red, "2" red]   & \vtwotwobarone \arrow[r, red, "2" red]& \vzerotwobarone \arrow[r, red, "2" red]& \vbartwotwobarone \arrow[r, red, "2" red]& \vbartwobartwobarone \arrow[dr, blue, "1" blue] &\\
 \vtwoonetwo \arrow[ur, red, "2" red]\arrow[dr, blue, "1" blue]&&&&&&&\vbaronebartwobarone\\
 &  \vtwoonetwo \arrow[r, red, "2" red]& \vzeroonetwo \arrow[r, red, "2" red]&\vbartwoonetwo \arrow[r, red, "2" red] \arrow[dr, blue, "1" blue]& \vbartwoonebartwo \arrow[r, blue, "1" blue] &\vbaroneonebartwo \arrow[r, blue, "1" blue] & \vbaronetwobarone \arrow[ur, red, "2" red]&\\
 &&&&\vbaroneonetwo \arrow[ur, red, "2" red]&&&
  \end{tikzcd}
\caption{The virtual crystal  $\vt^{BC}(\B(1^2|1))$ embedded into the type $C_2$-crystal $\C(3,1)$.}
\label{fig:virtualcrystal}
\end{figure} 

A nice consequence of the discussion above is that the \emph{orthogonal jeu de taquin} (BJDT), defined in \cite{lecouvey2003schensted}, can be reformulated in terms of the virtualization $\vt_{BC}$ and SJDT as the pullback of SJDT by $\vt_{BC}^{-1}=(\Split)^{-1}$.  

\begin{corollary}\label{cor:virt-rect-commute}
    Let $T$ be any orthogonal skew KN tableau with a potentially nonempty spin column. Then,
    \[
\vt_{BC} \circ \rect_B(T) = \rect_C \circ \vt_{BC}(T).
    \]
\end{corollary}

\begin{proof}
This is immediate from the arguments above where virtualization is first performed on words and then inserted. Namely, if $w(T)=x_1\cdots x_m$ denotes the reading word of $T$, then the following diagram commutes:
\begin{equation*}\label{BtoC}
\begin{tikzcd}
 x_1\otimes \cdots \otimes x_m\in\B(\omega_1^B)^{\otimes m}\arrow[r, "(\vt_{BC})^{\otimes m}"] \arrow[d, "B_n-insertion"] &  \vt_{BC}(x_1)\otimes \cdots \otimes \vt_{BC}(x_m) \in\B(2\omega^C_1)^{\otimes m} \arrow[d, "C_n-insertion", shift right=5ex ] \\
T=[  \emptyset \leftarrow x_{1} \leftarrow \cdots
 \leftarrow x_{m}] \in \KNb{n}{(\lambda)}   {\arrow[r, "\vt_{BC}"] }& \Split(T)=[  \emptyset \leftarrow \vt_{BC}(x_{1}) \leftarrow \cdots
 \leftarrow \vt_{BC}(x_{m})]\in  \KNc{n}{(2\lambda)}.
\end{tikzcd}
\end{equation*}
\end{proof}

\subsection{Applications to Keys and Orthogonal Evacuation} In what follows we give a  couple of combinatorial consequences of the work in Sections \ref{sec:crystals} and \ref{sec:demazure}. The first result is a combinatorial characterization of orthogonal keys. 

\begin{proposition}\label{prop:key}
    For any $T \in \KNb{n}{(\lambda)}$, we have that $T$ is an extremal weight vector if and only if the columns of $T$ are nested with no pairs $(i,\bar{i})$ appearing in any column for any $0\leq i\leq n$. 
\end{proposition}

\begin{proof}
     By Theorem \ref{thm:keys-virt-commute}, $T$ is a type $B_n$ left (resp. right) key if and only if $\Split(T)$ is a type $C_n$ left (resp. right) key. A symplectic key is a $C_n$ KN tableau where columns are nested and the letters $i$ and $\bar{i}$, for any $1\leq i\leq n$, do not appear simultaneously as entries in a column \cite{sa21a}. Hence  $T$ is a key in type $B_n$ if and only if the columns of $\Split(T)$  are nested. That is, the columns of $T$  are nested and  
     no column contains pairs of the form $(i, \bar{i})$ for any $i\neq 0$ and no entry equal to zero. Since zero pairs with itself, the result follows. 
\end{proof}

\begin{example}
Consider the tableau $\spintab{1\\\bar{2}}\tableau{\bar{2}} \in \B(1^2|1)$ in Figure \ref{fig:B2crystal}. Then indeed, no column contains any pairs with $\{\bar{2}\} \subset \{ 1 , \bar{2}\}$. A similar study of the other extremal weight vectors in $\B(1^2|1)$ yields similar conclusions.
\end{example}

Now, recall that by Proposition \ref{prop:virtualevac} we know that virtualization commutes with the Lusztig--Sch\"utzenberger involution $\xi$. That is, if $\vt$ is a virtualization from $\B$ to $\tilde{\B}$ then $\xi_{\B} = \vt^{-1} \xi_{\tilde{\B}} \vt$. Computing the Lusztig--Sch\"utzenberger involution on a given tableau coincides with so called \emph{evacuation}. 

For type $C$, Santos \cite{sa21a,sa21b}  developed a combinatorial method that directly computes this result without having to perform the involution on the entire crystal. In particular, given a symplectic tableau $T$, its evacuation $\evac_C(T)$ is given by first barring the unbarred entries and vice versa, followed by performing a 180 degree rotation of the tableau, and then rectifying the diagram via SJDT. 

For type $B_n$ such a direct combinatorial method is not known. Nonetheless, since evacuation of a tableau coincides exactly with computing its image under the Lusztig--Sch\"utzenberger involution, then setting $\vt=\vt_{BC}=\Split$ and $\xi_C = \evac_C$ (see \cite{sa21a}) we can compute orthogonal evacuation combinatorially by setting:
\begin{equation}\label{eq:evacB}
\evac_B:= (\Split)^{-1} \evac_C \Split
\end{equation}

\begin{example}\label{ex:evac}
Consider $T=\spintab{2\\\bar{1}}\tableau{0}\in \B(1^2|1)$ as in Figure \ref{fig:B2crystal}. Then, since $T= f_2f_1^2f_2\left(\spintab{1\\2}\tableau{1}\right)$
by looking at the complete crystal graph we can see that $\xi_B(T) = e_2e_1^2e_2\left(\spintab{\bar{2}\\\bar{1}}\tableau{\bar{1}}\right)= \bartwoonetwo$. 

On the other hand, computing $\evac_B(T)$ we find:
\[
\label{evacexample}
T=\zerotwobarone \xrightarrow{\operatorname{split}} \vzerotwobarone \xrightarrow{}\tableau{\bar{2}&\bar{2}&2\\1}\xrightarrow{180^0\text{-rotation}}\tableau{&&1\\2&\bar{2}&\bar{2}}\xrightarrow{\rect_C} \vbartwoonetwo\xrightarrow{\operatorname{split}^{-1}}\bartwoonetwo=\evac_B(T).
\]
\end{example}

It is natural to inquire what other types such a method can be applied. Theoretically, of course, the results are type independent. In practice, they are somewhat limited by the combinatorial tools available. 
In particular,  a jeu de taquin for type $D_n$ is not known. The pullback of SJDT, via the virtualization $\vt^{BC}$,  utilized to define a jeu de taquin in type $B_n$ does not work to define a jeu de taquin in type $D_n$ as shown in \cite{lecouvey2003schensted}.
Nonetheless, despite the fact that a type $D_n$ jeu de taquin is not available one has an $D_n$-insertion scheme for crystals \cite{KimInsertion,  lecouveythese, lecouvey2007combinatorics}. The evacuation in type $D_n$ should proceed in a similar way by applying the  longest element of the Weyl group of type $D_n$ to the entries of the tableau and replacing the SJDT by the appropriate insertion in type $D_n$. 

\section{Additional Examples}\label{sec:examples}
In this section we provide some explicit examples of how Theorems \ref{thm:extremal endpoints}, \ref{thm:Keys-Schutz-commute}, and \ref{thm:keys-virt-commute} can be used to compute the left/right keys of a given orthogonal tableau. 

\subsection{Applying Theorem \ref{thm:extremal endpoints} }Consider the $B_2$-crystal $B(1^2|1)$ in Figure \ref{fig:B2crystal} and its $3$-dilation $\Theta_3(B(1^2|1))$ in Figure \ref{fig:3dilate}. Then, by Theorem \ref{thm:extremal endpoints} we can compute the left and right key of any $b \in \B(1^2|1)$ by looking at its image in $\Theta_3(B(1^2|1))$. 

So let $T = \spintab{2\\\bar{1}}\tableau{0}= e_2^2e_1\left( \spintab{\bar{2}\\\bar{1}}\tableau{\bar{1}}\right)$. Then, since  
$
\Theta_3(T)= 
e_2^6e_1^3\left( \spintab{\bar{2}\\\bar{1}}\tableau{\bar{1}}^{\otimes 3}\right)
=\spintab{\bar{2}\\\bar{1}}\tableau{\bar{2}}\otimes \spintab{2\\ \bar{1}}\tableau{2}\otimes \spintab{2\\ \bar{1}}\tableau{2}
$
it follows from Theorem \ref{thm:extremal endpoints} that 
\[
K^+(T)=K^+\left(\spintab{2\\\bar{1}}\tableau{0}\right)=\spintab{\bar{2}\\\bar{1}}\tableau{\bar{2}}
\qquad\text{and}\qquad
K^-(T)=K^-\left(\spintab{2\\\bar{1}}\tableau{0}\right)=
\spintab{2\\\bar{1}}\tableau{{2}}.
\]

\subsection{Applying Theorem \ref{thm:Keys-Schutz-commute}}
We will now apply Theorem \ref{thm:Keys-Schutz-commute} in two different ways. Thus, from Example \ref{ex:evac} and the $3$-dilation of $\B(1^2|1)$ in Figure \ref{fig:3dilate} we have 
\[\xi_B(T) = e_2e_1^2e_2\left(\spintab{\bar{2}\\\bar{1}}\tableau{\bar{1}}\right), \qquad \text{and} \qquad
\Theta_3(\xi_B(T) ) = e_2^3e_1^6e_2^3\left(\spintab{\bar{2}\\\bar{1}}\tableau{\bar{1}}^{\otimes 3}\right) = \spintab{1\\\bar{2}}\tableau{\bar{2}}\otimes \spintab{1\\\bar{2}}\tableau{\bar{2}} \otimes \spintab{1\\{2}}\tableau{{2}}.\]
Thus,
$K^+(\xi_B(T))= \spintab{1\\\bar{2}}\tableau{\bar{2}} =e_1^2e_2\left(\spintab{\bar{2}\\\bar{1}}\tableau{\bar{1}}\right)$, and so $\xi_B(K^+(\xi_B(T))) = f_1^2f_2 \left( \spintab{1\\2}\tableau{2}\right) =\spintab{2\\\bar{1}}\tableau{2}$, which indeed coincides with $K^-(T)$.

On the other hand, in light of Corollary \ref{cor:virt-nospin} and Equation \eqref{eq:evacB} we can also perform this computation without the use of the full crystal graph. Namely, from Example \ref{ex:evac}

\[
\evac_B\left(K^+\left(\evac_B\spintab{2\\\bar{1}}\tableau{0}\right)\right) =\evac_B\left(K^+\left(\bartwoonetwo\right)\right)=    \operatorname{evac}^B\bartwoonebartwo=\spintab{2\\\bar{1}}\tableau{2}=K^-\left(\spintab{2\\\bar{1}}\tableau{0}\right).
\]

\subsection{Applying Theorem \ref{thm:keys-virt-commute}} Letting $\vt=\vt_{BC}$ and $T = \spintab{2\\\bar{1}}\tableau{0}$ as above, then using the information of the entire crystal graphs in Figures \ref{fig:B2crystal} and \ref{fig:virtualcrystal} and applying Corollary \ref{cor:virt-nospin}, we can verify that
\[
K^+(\vt_{BC}(T) ) = K^+(\Split(T) )=K^+\left( \tableau{{2}& 2&{\bar 2}\\ \bar 1}\right) = \tableau{ \bar 2\\\bar{1}}\tableau{\bar 2&\bar 2} =\Split \left( \spintab{\bar{2}\\\bar{1}}\tableau{\bar{2}} \right) = \vt_{BC}(K^+(T)).
\]

On the other hand, since by Theorem \ref{thm:keys-virt-commute} $K^+(T) = \vt_{BC}^{-1}\circ K^+(\vt_{BC}(T) )$ then, $K^+(T)$ can also be computed directly applying the procedure in \cite[Example 16]{sa21a, sa21b} and \cite[Example 5.2]{azenhas2024virtualkeys} that uses symplectic jeu de taquin (SJDT) and reverse symplectic jeu de taquin (RSJDT) on $\Split(T)$ and then virtualizing back. Namely, applying RSJDT to $\Split(T)=\tableau{{2}& 2&{\bar 2}\\ \bar 1}$ we obtain, 
\begin{align*}\label{RSJDT}
\tableau{{2}& 2&{\bar 2}\\ \bar 1}\stackrel{ RSJDT}{\longrightarrow} \tableau{{2}& 2&{\bar 2}\\ &\bar 1}   \stackrel{ RSJDT}{\longrightarrow} \tableau{& 2&{\bar 2}\\2 &\bar 1}  \stackrel{ RSJDT}{\longrightarrow} \tableau{{2}& 2&{\bar 2}\\ &\bar 1}   \stackrel{ RSJDT}{\longrightarrow} \tableau{& 2&{\bar 2}\\2 &&\bar 1}\stackrel{ RSJDT}{\longrightarrow} \tableau{& &{\bar 2}\\2 &2&\bar 1}.
\end{align*}
Then, if we consider the first and last skew-tableau in this sequence
\begin{align*}
 \tableau{{2}& 2&{\bar 2}\\ \bar 1}  {~\longrightarrow} ~\tableau{& &{\bar 2}\\2 &2&\bar 1}
\end{align*} 
we find that the right key $K^+(\Split(T))$ has columns $\tableau{{\bar 2}\\\bar 1}$ with multiplicity one and $\tableau{\bar 2}$ with multiplicity two (similarly, the left key $K^-(\Split(T))$ has columns $\tableau{{ 2}\\\bar 1}$ with multiplicity one and and $\tableau{ 2}$ with multiplicity two). Thus, by Theorem \ref{thm:keys-virt-commute} $K^+(T)$ is given by:
\[
K^+(T)=\Split^{-1}(K^+(\Split(T)))=\Split^{-1}(K^+\left(\tableau{{2}& 2&{\bar 2}\\ \bar 1}\right)) = \Split^{-1}\left(\tableau{{\bar 2}& \bar 2&{\bar 2}\\ \bar 1}\right)=\spintab{\bar 2\\\bar 1}\tableau{\bar 2}. 
\]
A similar computation yields, 
\[
K^-(T)=\Split^{-1}(K^-(\Split(T)))=\Split^{-1}\left(K^-\left(\tableau{{2}& 2&{\bar 2}\\ \bar 1}\right)\right)=\Split^{-1}\left(\tableau{{ 2}&  2&{ 2}\\ \bar 1}\right)=\spintab{ 2\\\bar 1}\tableau {2}.
\]

\subsection{Example of Theorem \ref{thm:extremal endpoints}} 
Finally, we offer an example of the differences between the $3$ and $6$-dilations of the $B_2$-crystal $\B(1^2|1)$ seen in Figures \ref{fig:3dilate} and \ref{fig:6-dilation}. Namely, by Theorem \ref{thm:extremal endpoints}, since the longest $i$-string in $\B(1^2|1)$ has length three, it suffices to compute the $3$-dilations in order to find the left/right keys of any given tableau $T \in \B(1^2|1)$. On the other hand, observe that in general the intermediate tensor factor in $\B(1^2|1)$ will not be an extremal weight vector. Take for instance, $T=f_1f_2\left(\spintab{1\\2}\tableau{1}\right) = \spintab{1\\ \bar{2}}\tableau{1}$. Then, $\Theta_3(T)= \spintab{2\\ \bar{1}}\tableau{2} \otimes \spintab{1\\\bar{2}}\tableau{2} \otimes \spintab{1\\\bar{2}}\tableau{1}$   
where both $\spintab{2\\ \bar{1}}\tableau{2}$ and $\spintab{1\\\bar{2}}\tableau{1}$ are extremal vectors, but $\spintab{1\\\bar{2}}\tableau{2} $ is not.

On the other hand, comparing this with 
$\Theta_6(T)= \spintab{2\\ \bar{1}}\tableau{2}^{\otimes 3} \otimes \spintab{1\\\bar{2}}\tableau{1}^{\otimes 3}$, we see that all vectors in this tensor product are indeed extremal. It is easy to see that this property holds for all vectors in $\B(1^2|1)$. Indeed, although we know of no references for the following, we suspect this property is true in general. Namely, that $\Theta_m(b) \in \cO(\lambda)^{\otimes m}$ all $b \in \B(\lambda)$  if and only if $\ell$ divides $m$ where $\ell$ is the is the least common multiple of the lengths of all the $i$-strings that originate at some extremal weight vector in $\B(\lambda)$. 

\begin{figure}[ht]
\begin{tikzcd}[ampersand replacement=\&,cramped,column sep=-1em, row sep=3em]
	\&\& \oneonetwo^{\otimes 3} \\
	\& \twoonetwo^{\otimes 3} \&\& \oneonebartwo^{\otimes 3} \\
	\& \bartwoonebartwo \otimes \twoonetwo^{\otimes 2} \&\& \twotwobarone \otimes \twoonebartwo \otimes\oneonebartwo \\
	\& \bartwoonebartwo^{\otimes 2}\otimes \twoonetwo \&\& \twotwobarone^{\otimes 3} \& \bartwobartwobarone \otimes \twoonebartwo \otimes \oneonebartwo \\
	\baronetwobarone \otimes \baroneonebartwo \otimes \twoonetwo \& \bartwoonebartwo^{\otimes 3} \&\& \bartwobartwobarone\otimes \twotwobarone^{\otimes 2} \\
	\& \baronetwobarone \otimes \baroneonebartwo \otimes \bartwoonebartwo \&\& \bartwobartwobarone^{\otimes 2}\otimes \twotwobarone \\
	\& \baronetwobarone^{\otimes 3} \&\& \bartwobartwobarone^{\otimes 3} \\
	\&\& \baronebartwobarone^{\otimes 3}
	\arrow["1^3"', from=1-3, to=2-2,blue]
	\arrow["2^3", from=1-3, to=2-4,red]
	\arrow["2^3"', from=2-2, to=3-2,red]
	\arrow["1^3", from=2-4, to=3-4,blue]
	\arrow["2^3"', from=3-2, to=4-2,red]
	\arrow["1^3"', from=3-4, to=4-4,blue]
	\arrow["2^3", from=3-4, to=4-5,red]
	\arrow["1^3"', from=4-2, to=5-1,blue]
	\arrow["2^3", from=4-2, to=5-2,red]
	\arrow["2^3"', from=4-4, to=5-4,red]
	\arrow["1^3", from=4-5, to=5-4,blue]
	\arrow["2^3"', from=5-1, to=6-2,red]
	\arrow["1^3", from=5-2, to=6-2,blue]
	\arrow["2^3", from=5-4, to=6-4,red]
	\arrow["1^3"', from=6-2, to=7-2,blue]
	\arrow["2^3", from=6-4, to=7-4,red]
	\arrow["2^3"', from=7-2, to=8-3,red]
	\arrow["1^3", from=7-4, to=8-3,blue]
\end{tikzcd}
\caption{The $3$-dilation $\Theta_3(\B(1^2|1)$ of the $B_2$-crystal $\B(1^2|1)$.}
\label{fig:3dilate}
\end{figure}

\begin{figure}
\begin{tikzcd}[ampersand replacement=\&,cramped,column sep=-1em, row sep=3em]
	\&\& \oneonetwo^{\otimes 6} \\
	\& \twoonetwo^{\otimes 6} \&\& \oneonebartwo^{\otimes 6} \\
	\& \bartwoonebartwo^{\otimes 2}\otimes \twoonetwo^{\otimes 4} \&\& \twotwobarone^{\otimes 3}\otimes\oneonebartwo^{\otimes 3} \\
	\& \bartwoonebartwo^{\otimes 4}\otimes \twoonetwo^{\otimes 2} \&\& \twotwobarone^{\otimes 6} \& \bartwobartwobarone^{\otimes 2}\otimes \twotwobarone \otimes \oneonebartwo^{\otimes 3} \\
	\baronetwobarone^{\otimes 3}\otimes \bartwoonebartwo\otimes \twoonetwo^{\otimes 2} \& \bartwoonebartwo^{\otimes 6} \&\& \bartwobartwobarone^{\otimes 2}\otimes \twotwobarone^{\otimes 4} \\
	\& \baronetwobarone^{\otimes 3}\otimes \bartwoonebartwo^{\otimes 3} \&\& \bartwobartwobarone^{\otimes 4}\otimes \twotwobarone^{\otimes 2} \\
	\& \baronetwobarone^{\otimes 6} \&\& \bartwobartwobarone^{\otimes 6} \\
	\&\& \baronebartwobarone^{\otimes 6}
	\arrow["1^6"', from=1-3, to=2-2,blue]
	\arrow["2^6", from=1-3, to=2-4,red]
	\arrow["2^6"', from=2-2, to=3-2,red]
	\arrow["1^6", from=2-4, to=3-4,blue]
	\arrow["2^6"', from=3-2, to=4-2,red]
	\arrow["1^6"', from=3-4, to=4-4,blue]
	\arrow["2^6", from=3-4, to=4-5,red]
	\arrow["1^6"', from=4-2, to=5-1,blue]
	\arrow["2^6", from=4-2, to=5-2,red]
	\arrow["2^6"', from=4-4, to=5-4,red]
	\arrow["1^6", from=4-5, to=5-4,blue]
	\arrow["2^6"', from=5-1, to=6-2,red]
	\arrow["1^6", from=5-2, to=6-2,blue]
	\arrow["2^6", from=5-4, to=6-4,red]
	\arrow["1^6"', from=6-2, to=7-2,blue]
	\arrow["2^6", from=6-4, to=7-4,red]
	\arrow["2^6"', from=7-2, to=8-3,red]
	\arrow["1^6", from=7-4, to=8-3,blue]
\end{tikzcd} 
\caption{The $6$-dilation $\Theta_6(\B(1^2|1)$ of the $B_2$-crystal $\B(1^2|1)$.}\label{fig:6-dilation}
\end{figure}

\bibliography{bibliography}
\bibliographystyle{alpha}

\end{document}